\documentclass[a4paper,11pt]{article}
\usepackage[applemac]{inputenc}
\usepackage{enumerate}
\usepackage{lmodern}
\usepackage[T1]{fontenc}
\usepackage{verbatim}
\usepackage{textcomp}
\usepackage[english]{babel}
\usepackage[a4paper,vmargin={3.5cm,3.5cm},hmargin={2.5cm,2.5cm}]{geometry}
\usepackage[font=sf, labelfont={sf,bf}, margin=1cm]{caption}
\usepackage[pdftex]{hyperref}
\usepackage[pdftex]{color,graphicx}

\usepackage{amsmath,amsfonts,amssymb,amsthm,mathrsfs}

\usepackage{stmaryrd}     
\usepackage{colonequals} 

\newtheorem{theorem}{Theorem}[]

\newtheorem{proposition}[theorem]{Proposition}
\newtheorem{lemma}[theorem]{Lemma}

\theoremstyle{definition}

\def\noi{\noindent}

\def\t{{\mathcal T}}

\def\ve{\varepsilon}
\def\wt{\widetilde}
\def\wh{\widehat}

\def\be{{\bf e}}
\def\bp{{\bf p}}
\def\bP{{\mathbf P}}
\def\bE{{\mathbf E}}

\def\la{\longrightarrow}
\def\T{{\mathbb T}}
\def\R{{\mathbb R}}
\def\P{{\mathbb P}}
\def\E{{\mathbb E}}
\def\N{{\mathbb N}}
\def\Z{{\mathbb Z}}

\def\t{\mathcal{T}}
\def\z{\mathcal{Z}}
\def\ts{\mathscr{T}}
\def\bt{\mathbf{t}}
\def\build#1_#2^#3{\mathrel{
\mathop{\kern 0pt#1}\limits_{#2}^{#3}}}

\def\ind{{\bf 1}_}

\title{The range of tree-indexed random walk}

\author{Jean-Fran\c cois LE GALL, Shen LIN\\
\\{\small Universit\'e Paris-Sud}}

\date{\tiny\today}

\begin{document}

\maketitle

\begin{abstract} We provide asymptotics for the range $R_n$ of a random walk 
on the $d$-dimensional lattice indexed by a random tree 
with $n$ vertices. Using Kingman's subadditive ergodic theorem, we prove under
general assumptions
that $n^{-1}R_n$ converges to a constant, and we give conditions ensuring that
the limiting constant is strictly positive. On the other hand, in dimension $4$ and in
the case of a symmetric random walk with exponential moments, we prove that
$R_n$ grows like $n/\log n$. We apply our results to asymptotics for the
range of branching random walk when the initial size of the population
tends to infinity. 


\smallskip
\noindent{\bf AMS 2010 Classification Numbers.} 60G50, 60J80
\end{abstract}

\section{Introduction}

The main goal of this work is to derive asymptotics for the
number of distinct sites of the lattice visited by a tree-indexed random walk.
Asymptotics for the range of an ordinary random walk on the 
$d$-dimensional lattice $\Z^d$ have been studied 
extensively since the pioneering work of Dvoretzky and Erd\"os
\cite{DE}. Consider for simplicity the case of a simple random walk
on $\Z^d$, and, for every integer $n\geq 1$, let ${\rm R}_n$
be the number of distinct sites of $\Z^d$ visited by the random walk 
up to time $n$. When $d\geq 3$, let $q_d>0$ be the probability that
the random walk never returns to its starting point. Then, 
\begin{enumerate}
\item[$\bullet$] if $d\geq 3$, 
$$\frac{1}{n}\,{\rm R}_n \build{\la}_{n\to\infty}^{\rm a.s.} q_d\,,$$
\item[$\bullet$] if $d=2$,
$$\frac{\log n}{n}\,{\rm R}_n \build{\la}_{n\to\infty}^{\rm a.s.} \pi\,,$$
\item[$\bullet$] if $d=1$,
$$n^{-1/2}\,{\rm R}_n \build{\la}_{n\to\infty}^{\rm (d)} \sup_{0\leq t\leq 1} B_t
- \inf_{0\leq t\leq 1} B_t\,,$$
\end{enumerate}
where $\build{\la}_{}^{\rm(d)}$ indicates convergence in distribution and
$(B_t)_{t\geq 0}$ is a standard linear Brownian motion. The cases
$d\geq 3$ and $d=2$ were obtained in \cite{DE}, whereas the case $d=1$
is a very easy consequence of Donsker's invariance theorem (see e.g.\,\cite{JP2}). The preceding
asymptotics have been extended to much more general random walks. In
particular, for any random walk in $\Z^d$, an application of Kingman's subadditive ergodic theorem \cite{King}
shows that the quantity ${\rm R}_n/n$ converges a.s.~to the probability that the random walk
does not return to its starting point (which is positive if the random walk is transient). See also \cite{JP2} for the almost sure convergence
of the (suitably normalized) range of an arbitrary recurrent random walk in the plane, \cite{JP1} for a
central limit theorem for the range of transient random walk, \cite{LGRW} for a non-Gaussian
central limit theorem in the plane and \cite{LGR} for a general study of the range of random
walks in the domain of attraction of a stable distribution.

In the present work, we discuss similar asymptotics
for tree-indexed random walk. We consider (discrete) plane trees, which are
rooted ordered trees that can be viewed as describing the genealogy of
a population starting with one ancestor or root, which is usually denoted by 
the symbol $\varnothing$. Given such a tree $\t$ and a probability measure 
$\theta$ on $\Z^d$, we can consider the random
walk with jump distribution $\theta$ indexed by the tree $\t$. This means that we assign a 
(random) spatial
location $Z_{\t}(u)\in \Z^d$ to every vertex $u$ of $\t$, in the following way. First, the
spatial location $Z_{\t}(\varnothing)$ of the root is the origin of $\Z^d$.
Then, we assign
independently to every
edge $e$ of the tree $\t$ a random variable $X_e$ distributed according to $\theta$, and we let 
the spatial location $Z_{\t}(u)$  of the vertex $u$ be the sum of the
quantities $X_e$ over all edges $e$ belonging to the simple path 
from $\varnothing$ to $u$ in the tree. The number of distinct spatial locations 
is called the range of the tree-indexed random walk $Z_{\t}$.

Let us state a particular case of our results.

\begin{theorem}
\label{treeSRW}
Let $\theta$ be a probability distribution on $\Z^d$, which is symmetric
and has finite support. Assume that $\theta$ is not supported on a 
strict subgroup of $\Z^d$.
For every integer $n\geq 1$, let $\t_n$ be a random tree uniformly distributed
over all plane trees with $n$ vertices. Conditionally given
$\t_n$, let $Z_{\t_{n}}$ be a random walk with jump distribution
$\theta$ indexed by $\t_n$,
and let $\mathcal{R}_n$ stand for the range of $Z_{\t_{n}}$. Then,
\begin{enumerate}
\item[$\bullet$] if $d\geq 5$, 
$$\frac{1}{n}\,\mathcal{R}_n \build{\la}_{n\to\infty}^{\rm (P)} c_\theta\,,$$
where $c_\theta>0$ is a constant depending on $\theta$, and $\build{\la}_{}^{\rm (P)}$ indicates convergence in probability;
\item[$\bullet$] if $d=4$,
$$\frac{\log n}{n}\,\mathcal{R}_n \build{\la}_{n\to\infty}^{\rm (P)} 
8\,\pi^2\,\sigma^4\,,$$
where $\sigma^2=({\rm det}\,M_\theta)^{1/4}$, with $M_\theta$ denoting the covariance matrix of $\theta$;
\item[$\bullet$] if $d\leq 3$, 
$$n^{-d/4}\,\mathcal{R}_n \build{\la}_{n\to\infty}^{\rm (d)} 
c_\theta\,\lambda_d({\rm supp}(\mathcal{I}))\,,$$
where $c_\theta=2^{d/4}({\rm det} M_{\theta})^{1/2}$ is a constant depending on $\theta$, and $ \lambda_d({\rm supp}(\mathcal{I}))$ stands for the 
Lebesgue measure of the support of the random measure on $\R^{d}$ known as
ISE (Integrated Super-Brownian Excursion). 
\end{enumerate}
\end{theorem}

Notice the obvious analogy with the results for the range of 
(ordinary) random walk
that were recalled above. At an intuitive level, $\mathcal{R}_n$
is likely to be smaller than the range ${\rm R}_n$ of ordinary random walk, because 
one expects many more self-intersections in the tree-indexed 
case. This is reflected in the fact that the ``critical dimension''
is now $d=4$ instead of $d=2$. In the same way as $d=2$ is critical for the recurrence
of random walk on $\Z^d$, one may say that $d=4$ is critical for the 
recurrence of tree-indexed random walk, in the sense that for 
random walk indexed by a ``typical'' large tree of size $n$, the number of returns to
the origin will grow logarithmically with $n$. Furthermore, one may notice that 
the set of all spatial locations of $\t_n$ is contained in the ball of radius $Cn^{1/4}$
centered at the origin, with a probability
close to $1$ 
if the constant $C$ is sufficiently large (see 
Janson and Marckert \cite{JM} or Kesten \cite{Kes} in a slightly different setting),
so that the range $\mathcal{R}_n$ is at most of order $n^{d/4}$ in dimension $d\leq 3$. 
We finally mention that the
limiting constant $c_\theta$ in dimension $d\geq 5$ can again be interpreted 
as a probability of no return to the origin for 
random walk indexed by a certain infinite random tree: See Section 2 below
for more details. 

Let us emphasize that asymptotics of the type of Theorem~\ref{treeSRW}
hold in a much more general setting. Firstly, it is enough to
assume that the jump distribution $\theta$ is centered and
has sufficiently high moments (a little more is needed when $d=4$). Our argument to get the case
$d\geq 5$ of Theorem~\ref{treeSRW} relies on an application of Kingman's
subadditive ergodic theorem, which gives the convergence of 
$\frac{1}{n}\mathcal{R}_n$ to a (possibly vanishing) constant in any dimension $d$, without 
any moment assumption on $\theta$. Secondly, in all cases except the critical dimension $d=4$,
 we can handle more general random trees.
Our methods apply to Galton-Watson trees with an offspring distribution
having mean one and finite variance, which are conditioned to have exactly 
$n$ vertices. In the special case where the offspring distribution 
is geometric with parameter $1/2$, we recover uniformly distributed
plane trees, but the setting of conditioned Galton-Watson
trees includes other important ``combinatorial trees'' such
as binary trees or Cayley trees (see e.g.~\cite{LG1}). Some of our results
even hold for an offspring distribution with infinite variance in the
domain of attraction of a stable distribution.

In the present work, we deal with the cases $d\geq 5$ and $d=4$
of Theorem~\ref{treeSRW}, and the extensions that have just been 
described. The companion paper~\cite{LGL} will address the 
``subcritical'' case $d\leq 3$, which involves rather different 
methods and is closely related to the invariance principles 
connecting branching random walk with super-Brownian motion.

Let us turn to a more precise description of our 
main results and of our methods. In Section 2 below, we discuss the
convergence of $\frac{1}{n}\mathcal{R}_n$ in a general setting. The basic ingredient of the proof is the introduction of
a suitable probability measure on a certain set of infinite trees. Roughly speaking, for any
offspring distribution $\mu$ with mean one, we construct a random
infinite tree consisting of an infinite ``spine'' and, for each node of the spine,
of a random number of Galton-Watson trees with offspring distribution
$\mu$ that branch off the spine at this node. For a more precise description, see subsection~\ref{invariant-infinite-trees}. The law of this infinite tree
turns out to be invariant under a shift transformation, 
which basically involves
re-rooting the tree at the first vertex (in lexicographical order) that does not
belong to the spine. If we consider 
a random walk (with an arbitrary jump distribution $\theta$) indexed by
this infinite tree, the number of distinct locations of the random walk at
the first $n$ vertices of the infinite tree yields a subadditive process $R_n$,
to which we can apply Kingman's theorem in order to get the almost sure 
convergence of $\frac{1}{n}R_n$ to a constant (Theorem~\ref{subaddi}). One then needs to discuss the positivity of the limiting constant,
and this leads to conditions depending both on the offspring distribution 
$\mu$ and
on the jump distribution $\theta$. More precisely, we 
give a criterion (Proposition~\ref{suffcond}) involving the Green function of the random walk
and the generating function of $\mu$,
which ensures that the limiting constant is positive. In the case
when
$\mu$ has finite variance and if the jump distribution $\theta$ is centered
(with sufficiently high moments), this criterion is satisfied 
if $d\geq 5$. 
The preceding line of  reasoning is of course very similar to the 
classical application
of Kingman's theorem to the range of ordinary random walk. In the
present setting however, additional ingredients are needed to transfer
the asymptotics from the case of the infinite random tree to a single
Galton-Watson tree conditioned to have $n$ vertices. At this point we need
to assume that the offspring distribution $\mu$ has finite variance or
is in the domain of attraction of a stable distribution, so that we can use
known results~\cite{DLG} on the scaling limit of the height process 
associated with a sequence of Galton-Watson trees
with offspring distribution $\mu$: Applying these
results to the sequence of trees that branch off the spine of the infinite tree
yields information about the ``large'' trees in the sequence, which is 
essentially what we need to cover the case of a single Galton-Watson tree
conditioned to be large (Theorem~\ref{mainsuper}). The case $d\geq 5$ of
Theorem \ref{treeSRW} follows as a special case of the results in Section 2.

Section 3, which is the most technical part of the paper, is devoted to
the proof of a generalized version of the case $d=4$
of Theorem~\ref{treeSRW} (Theorem~\ref{rangecritisnake}). We restrict our attention to the
case when the offspring distribution is geometric with parameter $1/2$,
and we assume that the jump distribution $\theta$ is symmetric with small
exponential moments. While the symmetry assumption can presumably
be weakened without too much additional work, the existence of exponential
moments is used at a crucial point of our proof where we rely on 
the multidimensional extension of the celebrated Koml\' os-Major-Tusn\' ady strong
invariance principle. Our approach is based on the path-valued Markov chain
called the discrete snake. In our setting, this process, which 
we denote by $(W_n)_{n\geq 0}$, takes values in the space 
of all infinite paths $w:(-\infty,\zeta]\cap \Z \la \Z^4$, where 
$\zeta=\zeta(w)\in \Z$ is called the lifetime of $w$. If $\zeta_n$ denotes the
lifetime of $W_n$, the process $(\zeta_n)_{n\geq 0}$ evolves like
simple random walk on $\Z$. Furthermore, if $\zeta_{n+1}=\zeta_n -1$,
the path $W_{n+1}$ is obtained by restricting $W_n$
to the interval $(-\infty,\zeta_n-1]\cap \Z$, whereas if $\zeta_{n+1}=\zeta_n +1$,
the path $W_{n+1}$ is obtained by adding to $W_n$ one step 
distributed according to $\theta$. We assume that the initial value 
$W_0$ is just a path (indexed by negative times) of the random walk with jump distribution $\theta$
started from the origin.
Then the values of the discrete snake generate
a random walk indexed by an infinite random tree, which corresponds,
in the particular case of the geometric offspring distribution, to
the construction developed in Section 2. Note however that, in contrast with Section 2, the Markovian properties of the
discrete snake play a~very important role in Section 3. A key estimate
(Proposition~\ref{firstesti}) states that the probability that the ``head of
the discrete snake'' (that is the process $(W_k(\zeta_k))_{k\geq 0}$) does not return to
the origin before time $n$ behaves like $c/\log n$ for a certain
constant $c$. This is analogous to the well-known asymptotics for
the probability that random walk in the plane does not come back to its
starting point before time $n$, but the proof, which is developed in
subsection~\ref{mainest}, turns out to be much more
involved in our setting. The main result of Section 3 (Theorem \ref{rangecritisnake}) gives
the case $d=4$ of Theorem \ref{treeSRW} under slightly more general assumptions.

Finally, Section 4 applies the preceding results to asymptotics for the range of 
a branching random walk in $\Z^d$, $d\geq 4$, when the size of the initial 
population tends to infinity. This study
is related to the recent work of Lalley and Zheng \cite{LZ}
who discuss the number of distinct sites occupied by a 
nearest neighbor branching random walk in $\Z^d$
at a fixed time. Note that the genealogical structures of descendants of the different
initial particles are described by independent Galton-Watson trees, which
makes it possible to apply our results about the range of 
tree-indexed random walk. Still one needs to verify that points
that are visited by the descendants of two distinct initial particles give a 
negligible contribution in the limit. The analogous problem for low dimensions
$d\leq 3$ will be addressed in \cite{LGL}.

\medskip
\noi{\it Notation.} We use the notation $\llbracket a, b\rrbracket \colonequals [a,b]\cap \Z$
for $a,b\in \Z$, with $a\leq b$. Similarly, $\rrbracket -\infty,a\rrbracket \colonequals(-\infty,a]\cap \Z$
for $a\in \Z$. For any finite set $A$, $\#A$ denotes the cardinality of $A$.

\section{Linear growth of the range}
\label{supercri}


\subsection{Finite trees}
\label{fitree}

We use the standard formalism for plane trees. We set
$$\mathcal{U} \colonequals\bigcup_{n=0}^\infty \N^n,$$
where $\N=\{1,2,\ldots\}$ and $\N^0=\{\varnothing\}$. If $u=(u_1,\ldots,u_n)\in\mathcal{U}$,
we set $|u|=n$ (in particular $|\varnothing|=0$).
We write $\prec$ for the lexicographical order on $\mathcal{U}$, so that $\varnothing \prec 1 
\prec (1,1) \prec 2$ for instance. 

If $u,v\in\mathcal{U}$, $uv$ stands for the
concatenation of $u$ and $v$. In particular $\varnothing u=u \varnothing =u$. 
The genealogical (partial) order $\ll$ is then defined by 
saying that $u \ll v$ if and only if 
$v=uw$ for some $w\in \mathcal{U}$.

A plane tree (also called rooted ordered tree) $ \t$ is a finite subset of $\mathcal{U}$
such that the following holds:
\begin{enumerate}
\item[(i)] $\varnothing\in \t$.

\item[(ii)] If $u=(u_1,\ldots,u_n)\in \t\backslash\{\varnothing\}$ then 
$\widehat  u:=(u_1,\ldots,u_{n-1})\in \t$.

\item[(iii)] For every $u=(u_1,\ldots,u_n)\in\t$, there exists an integer $k_u(\t)\geq 0$
such that, for every $j\in\N$,  $(u_1,\ldots,u_n,j)\in\t$ if and only if $1\leq j\leq
k_u(\t)$.
\end{enumerate}

The notions of a child and a parent of a vertex of $\t$ are defined in an obvious way.
The quantity $k_u(\t)$ in (iii) is the number of children of $u$ in $\t$.
If $u\in \t$, we write $[\t]_u=\{v\in\mathcal{U}:uv\in\t\}$, which corresponds to the
subtree of descendants of $u$ in $\t$. We denote the set of all plane trees by $\T_f$.

Throughout this work, we consider a probability measure $\mu$ on $\Z_+$, which is critical in the sense that
$$\sum_{k=0}^\infty k\,\mu(k)=1.$$
We exclude the degenerate case where $\mu(1)=1$. The law of the Galton-Watson tree with offspring
distribution $\mu$ is a probability measure on the space $\T_f$, which
we denote by $\Pi_\mu$ (see e.g.~\cite[Section 1]{LG1}). 

We also consider a 
random walk $S=(S_k)_{k\geq 0}$ in $\Z^d$, with jump distribution  $\theta$.
We assume that $S
$ is adapted (i.e.~$\theta$ is not supported on a strict subgroup
of $\Z^d$).
It will be convenient to assume that the random walk $S$ 
starts from $x$ under the
probability measure $P_x$, for every $x\in \Z^d$. 

A ($d$-dimensional) spatial tree is a pair $(\t,(z_u)_{u\in\t})$
where $\t\in \T_f$ and $z_u\in \Z^d$ for every $u\in \t$. Let $\T_f^*$
be the set of all spatial trees. We write
$\Pi^*_{\mu,\theta}$ for the probability distribution on $\T_f^*$ under which $\t$ is distributed according to $\Pi_\mu$ and,
conditionally on $\t$, the ``spatial locations'' $(z_u)_{u\in\t}$ are distributed as random walk indexed by $\t$,
with jump distribution $\theta$, and started from $0$ at the root $\varnothing$
(see the definition given in Section 1). We then set
$$a_{\mu,\theta} \colonequals \Pi^*_{\mu,\theta}(z_u\not =0,\,\forall u\in \t\backslash\{\varnothing\}),$$
and, for every $y\in\Z^d$,
$$h_{\mu,\theta}(y) \colonequals \Pi^*_{\mu,\theta}(z_u\not = - y,\,\forall u\in \t).$$
Notice that $a_{\mu,\theta}>0$, simply because with positive probability
a tree distributed according to $\Pi_\mu$ consists only of the root.

\subsection{Infinite trees}
\label{infitree}

We now introduce a certain class of infinite trees. Each 
tree in this class will consist of an infinite ray or spine starting from the root, and finite subtrees
branching off every node of this infinite ray. 
We label the vertices of the infinite ray by nonpositive integers $0,-1,-2,\ldots$.
The reason
for labelling the vertices of the spine by negative integers comes from the fact that
$-1$ is viewed as the parent of $0$, $-2$ as the parent of $-1$, and so on.

More precisely, we consider the set
$$\mathcal{V} \colonequals \Z_- \times \mathcal{U}$$
where $\Z_-=\{0,-1,-2,\ldots\}$. For every $j\in\Z_-$, we identify the element
$(j,\varnothing)$ of $\mathcal{V}$ with the integer $j$, and we thus view 
$\Z_-$ as a subset of $\mathcal{V}$. We define the lexicographical order on
$\mathcal{V}$ as follows. If $j,j'\in \Z_-$, 
we have $j\prec j'$ if and only if $j\leq j'$.  If $u\in \mathcal{U}\backslash\{\varnothing\}$,
we have always $j'\prec (j,u)$. If $u,u'\in \mathcal{U}\backslash\{\varnothing\}$,
we have $(j,u)\prec (j',u')$ if either $j>j'$, or
$j=j'$ and $u\prec u'$. The genealogical (partial) order $\ll$ on $\mathcal{V}$ is defined in
an obvious way: in agreement with the preceding heuristic
interpretation, the property $j\ll j'$ for $j,j'\in\Z_-$ holds if and only if $j\leq j'$. 

Let $\t$ be a subset of $\mathcal{V}$ such that $\Z_-\subset \t$. For every
$j\in\Z_-$, we set
$$\t_j:=\{u\in \mathcal{U}: (j,u)\in \t\}.$$
We say that $\t$ is an {\it infinite tree} if, for every $j\in \Z_-$, $\t_j$
is a (finite) plane tree, and furthermore $\t\backslash \Z_-$ is
infinite. We write $\T$ for the set of all infinite trees. By convention,
the root of an infinite tree $\t$ is the vertex $0$. Clearly, $\t$ is determined by
the collection $(\t_j)_{j\in \Z_-}$. Note that the lexicographical order of vertices
corresponds to the order of visit when one ``moves around'' the tree in clockwise
order, starting from the ``bottom'' of the spine and assuming that the
``subtrees'' $\t_j$ are drawn on the right side of the spine, as in Fig.1.

We next define a shift transformation $\tau$ on the space $\T$. 
Starting from an infinite tree $\t$, its image $\tau(\t)=\t'$
is obtained informally as follows. We look for the first vertex (in
lexicographical order) of $\t\backslash \Z_-$. Call this vertex $v$. 
We then ``re-root'' the tree $\t$ at $v$ and, in the case when $v$ is not
a child of $0$ (or equivalently if $\t_0=\{\varnothing\}$), we remove the vertices 
of the spine that are strict descendants of the parent of $v$. 

For a more formal definition, let $k\in \Z_-$ be the unique integer such that $v\in\t_k$ (necessarily, $v=(k,1)$). Then, $\t'$ is determined by requiring that: 
\begin{itemize}
\item  $\t'_j =\t_{j+k+1}$ if $j\leq -2$;
\item  $\t'_0= [\t_k]_1$;
\item  $\t'_{-1}$ is the unique plane tree such that there exists a bijection from 
$\t_k\backslash \{u\in\t_k : 1\ll u\}$ onto $\t'_{-1}$ that preserves both the 
lexicographical order and the genealogical order. 
\end{itemize}

\begin{figure}[!htbp]
 \begin{center}
 \includegraphics[width=14.5cm]{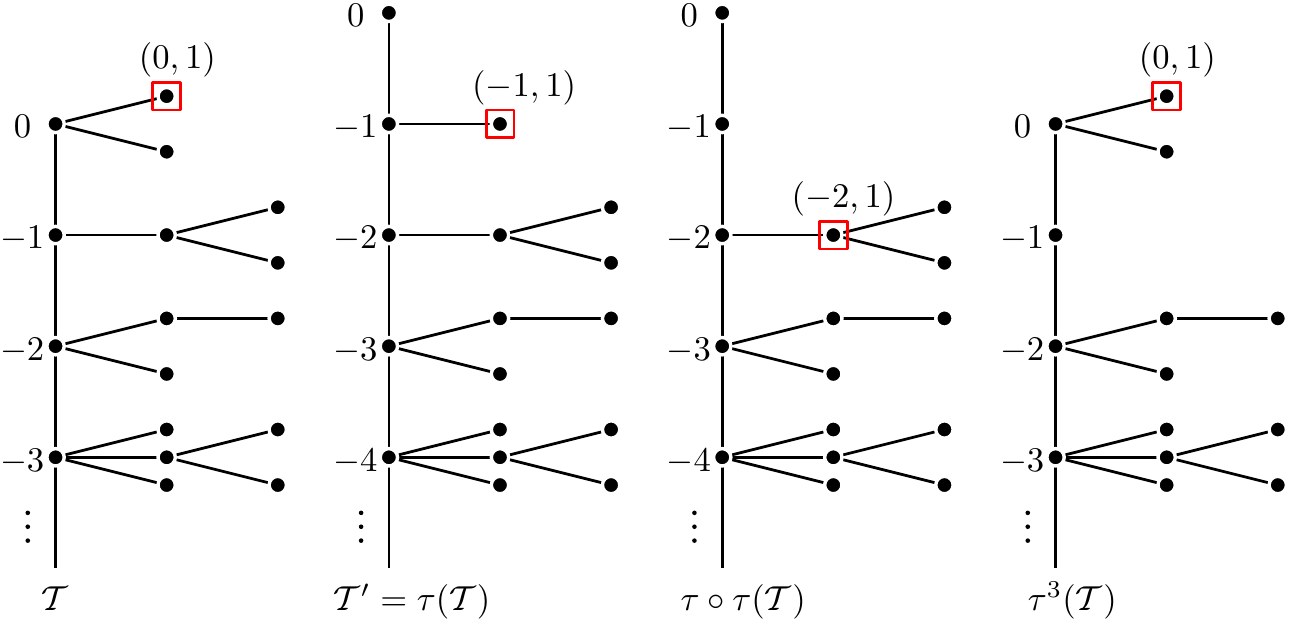}
 \caption{\textsf{The first 3 iterations of the shift transformation on an infinite tree $\t$. At each step, the marked vertex will become the new root after the shift.} \label{fig-shift}}
 \end{center}
\end{figure}

Fig.~\ref{fig-shift} explains the construction of $\t'$ better than the formal definition.

\subsection{The invariant measure on infinite trees}
\label{invariant-infinite-trees}
Let $\bP_\mu$ be the probability measure on $\T$ that is determined by
the following conditions. Under $\bP_\mu(\mathrm{d}\t)$,
\begin{itemize}
\item the trees $\t_0,\t_{-1},\t_{-2},\ldots$ are independent;
\item $\t_0$ is distributed according to $\Pi_\mu$;
\item for every integer $j\leq -1$, 
$$\bP_\mu(k_\varnothing(\t_j)=n)= \mu([n+1,\infty)),$$
for every $n\geq 0$; furthermore, conditionally on $k_\varnothing(\t_j)=n$, the trees
$[\t_j]_1,[\t_j]_2,\ldots,[\t_j]_n$ are independent and distributed according to $\Pi_\mu$.
\end{itemize}
Notice that $\sum_{n\geq 0}\mu([n+1,\infty))=1$ due to the criticality of the probability measure $\mu$. The reason for introducing the probability measure $\bP_\mu$ comes from the next proposition.

\begin{proposition}
\label{invariantmeas}
The probability measure $\bP_\mu$ is invariant under the shift $\tau$.
\end{proposition}

\proof Suppose that $\t$ is distributed according to $\bP_\mu$ and set
$\t'=\tau(\t)$ as above. We need to verify that $\t'$ is also distributed according to $\bP_\mu$,
or equivalently that the trees $\t'_0,\t'_{-1},\ldots$ satisfy the same properties 
as $\t_0,\t_{-1},\ldots$ above. The key point is to calculate the
distribution of $(k_\varnothing(\t'_j),j\leq 0)$.
Fix an integer $p\geq 1$, and let $n_0,n_1,\ldots, n_p\in \Z_+$. Also let 
$k$ be the element of $\Z_-$ determined as in the definition of
$\t'=\tau(\t)$ at the end of subsection \ref{infitree}. The event
$$\{k=0\}\cap \big\{k_\varnothing(\t'_0)=n_0,k_\varnothing(\t'_{-1})=n_1,\ldots,k_\varnothing(\t'_{-p})=n_p\big\}$$
holds if and only if we have
$$k_\varnothing(\t_0)=n_1+1,\;k_1(\t_0)= n_0,\;k_\varnothing(\t_{-1})=n_2,\ldots,\; k_\varnothing(\t_{-p+1})=n_p,$$
which occurs with probability
$$\mu(n_1+1)\mu(n_0)\mu([n_2+1,\infty))\ldots \mu([n_p+1,\infty)).$$
Let $\ell\in \Z_-\backslash\{0\}$. Similarly, the event
$$\{k=\ell\}\cap \big\{k_\varnothing(\t'_0)=n_0,k_\varnothing(\t'_{-1})=n_1,\ldots,k_\varnothing(\t'_{-p})=n_p\big\}$$
holds if and only if we have
$$k_\varnothing(\t_0)=0,\ldots,\, k_\varnothing(\t_{\ell +1})=0,\,
k_\varnothing(\t_{\ell})=n_1+1,\,k_1(\t_\ell)=n_0,\,
k_\varnothing(\t_{\ell-1})=n_2,\ldots,\, k_\varnothing(\t_{\ell-p+1})=n_p,$$
which occurs with probability
$$\mu(0)\mu([1,\infty))^{-\ell-1}\mu([n_1+2,\infty))\mu(n_0)\mu([n_2+1,\infty))\ldots
\mu([n_p+1,\infty)).$$
Summarizing, we see that the event 
$$\big\{k_\varnothing(\t'_0)=n_0,k_\varnothing(\t'_{-1})=n_1,\ldots,k_\varnothing(\t'_{-p})=n_p\big\}$$
has probability
\begin{align*}
&\mu(n_0)\mu(n_1+1)\mu([n_2+1,\infty))\ldots \mu([n_p+1,\infty))\\
&\ +\mu(n_0)\mu(0)\Big(\sum_{\ell=-1}^{-\infty}
\mu([1,\infty))^{-\ell-1}\Big)\mu([n_1+2,\infty))\mu([n_2+1,\infty))\ldots
\mu([n_p+1,\infty))\\
&= \mu(n_0) \mu([n_1+1,\infty))\mu([n_2+1,\infty))\ldots
\mu([n_p+1,\infty)),
\end{align*}
as desired. An immediate generalization of the preceding argument shows that, if
$\bt_0$ and $\bt_{j,i}$, $1\leq j\leq p$, $1\leq i\leq n_j$ are given plane trees,
 the event
$$\big\{k_\varnothing(\t'_{-1})=n_1,\ldots,k_\varnothing(\t'_{-p})=n_p\big\} \cap \big\{\t'_0=\bt_0\big\}
\cap \Bigg(\bigcap_{j=1}^p \Big( \bigcap_{i=1}^{n_j} \{[\t'_{-j}]_i=\bt_{j,i}\}\Big)\Bigg)$$
has probability
$$\mu([n_1+1,\infty))\mu([n_2+1,\infty))\ldots
\mu([n_p+1,\infty))\times \Pi_\mu(\bt_0)
\times \prod_{j=1}^p\Big(\prod_{i=1}^{n_j} \Pi_\mu(\bt_{j,i})\Big).$$
This completes the proof. \endproof

\subsection{Random walk indexed by the infinite tree}
\label{sec:random-walk-indexed}

Let $\t\in \T$. The definition of random walk indexed by $\t$
requires some extra care because we need to specify the orientation of edges:  The (oriented) edges of $\t$ are all pairs $(x,y)$ of elements 
of $\t$ such that there exists $j\in \Z_-$ such that 
\begin{itemize}
\item either $x=(j,u)$, $y=(j,v)$, where $u,v\in\t_j$ and $u$ is the parent of $v$;
\item or $x=j-1$, $y=j$.
\end{itemize}
See Fig.\ref{fig-orientation}. We write $\mathcal{E}(\t)$ for the collection of all oriented edges of $\t$. 
The random walk indexed by $\t$ is a collection $(Z_\t(u))_{u\in \t}$ of random variables 
with values in $\Z^d$, such that
$Z_\t(0)=0$ and the random variables $(Z_\t(y)-Z_\t(x))_{(x,y)\in\mathcal{E}(\t)}$
are independent and distributed according to $\theta$. Let $P_{(\t)}$ stand for the
distribution of the collection $(Z_\t(u))_{u\in \t}$.

Let $\T^*$ be the set of all pairs $(\t,(z_u)_{u\in \t})$ where $\t\in \T$
and $z_u\in \Z^d$ for every $u\in\t$. 
We define a probability measure 
$\bP^*_{\mu,\theta}$ on $\T^*$ by declaring that $\bP^*_{\mu,\theta}$
is the law of the random pair $(\ts,(\mathscr{Z}_u)_{u\in \ts})$ where 
$\ts$ is distributed according to $\bP_\mu$ and conditionally on
$\ts=\t$, $(\mathscr{Z}_u)_{u\in \ts}$ is distributed according to $P_{(\t)}$.

We next define a shift transformation $\tau^*$ on $\T^*$. 
For $(\t,(z_u)_{u\in \t})\in \T^*$, we set $\tau^*(\t,(z_u)_{u\in \t})
=(\t',(z'_u)_{u\in\t'})$, where $\t'=\tau(\t)$ and the spatial locations of vertices of 
$\t'$ (which may be viewed as a subset of $\t$) are obtained by
shifting all original locations $z_u$ so that the location of the root 
of $\t'$ is again $0$. More precisely, if $k\in\Z_-$ is defined as 
above in the definition of $\t'=\tau(\t)$, there is a unique 
bijection $\phi_\t$ from $\t'$ onto $\t\backslash\{k+1,k+2,\ldots,0\}$
that maps $0$ to $(k,1)$ and preserves both the lexicographical order
and the genealogical order, and we set
$$z'_u= z_{\phi_\t(u)} - z_{\phi_\t(0)}$$
for every $u\in\t'$.

\begin{proposition}
\label{invariantmeas2}
The probability measure $\bP^*_{\mu,\theta}$ is invariant under $\tau^*$.
\end{proposition}

This is an easy consequence of Proposition~\ref{invariantmeas}
and the way the spatial positions are constructed. We leave the
details to the reader. 

\begin{figure}[!htbp]
 \begin{center}
 \includegraphics[width=6cm]{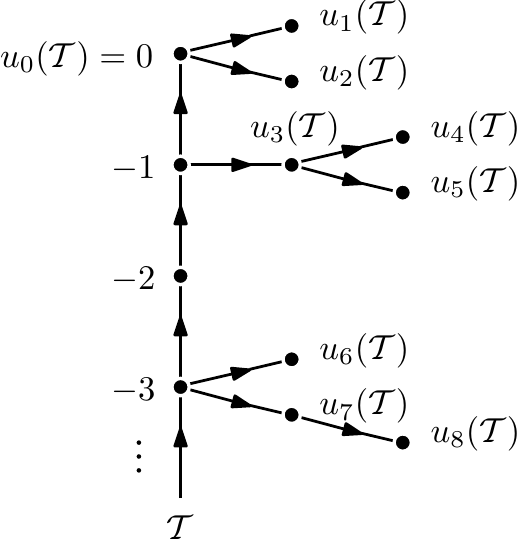}
 \caption{\textsf{The orientation of edges of $\t$, and the sequence
 $u_0(\t),u_1(\t),u_2(\t),\ldots$} \label{fig-orientation}}
 \end{center}
\end{figure}

Let $\t^*=(\t,(z_u)_{u\in \t})\in \T^*$. We define a sequence 
$(u_i(\t))_{i\geq 0}$ of elements of $\t$ as follows. First,
$u_0(\t)=0$ is the root of $\t$. Then 
$u_1(\t),u_2(\t),\ldots$ are all elements of 
$\t\backslash \Z_-$ listed in lexicographical order (see Fig.\ref{fig-orientation}). Finally,
we set, for every integer $n\geq 1$,
$$R_n(\t^*) \colonequals\#\{z_{u_0(\t)},z_{u_1(\t)},\ldots, z_{u_{n-1}(\t)}\}.$$

Recall the notation $a_{\mu,\theta}$ and $h_{\mu,\theta}$ introduced at the end of
subsection~\ref{fitree}. 

\begin{theorem}
\label{subaddi}
We have
$$\frac{R_n}{n} \build{\la}_{n\to\infty}^{} c_{\mu,\theta},\quad\bP^*_{\mu,\theta}\hbox{-a.s.}$$
where the limiting constant $c_{\mu,\theta}\in[0,1]$ may be defined as
$$c_{\mu,\theta} = a_{\mu,\theta}\;
E\Bigg[\prod_{j=1}^\infty \Phi_{\mu,\theta}(-S_j)\Bigg],$$
with
$$\Phi_{\mu,\theta}(x) = \sum_{k=0}^\infty \mu([k+1,\infty))
\Big(\sum_{y\in\Z^d} \theta(y)\,h_{\mu,\theta}(x+y)\Big)^k,$$
for every $x\in \Z^d$. 
\end{theorem}

\proof Set $\tau^*_n = (\tau^*)^n$ for every integer $n\geq 1$. We claim
that, for every $n,m\geq 1$,
$$R_{n+m}\leq R_n + R_m\circ \tau^*_n.$$
Indeed, $R_n(\t^*)$
is the number of distinct elements among $z_{u_0(\t)},z_{u_1(\t)},\ldots, z_{u_{n-1}(\t)}$, and similarly $R_{n+m}(\t^*)$
is the number of distinct elements among $z_{u_0(\t)},z_{u_1(\t)},\ldots, z_{u_{n+m-1}(\t)}$. On the other hand, from the construction of the shift transformation,
it is fairly easy to verify that $R_m\circ \tau^*_n(\t^*)$ is the number of
distinct elements among $z_{u_n(\t)},z_{u_{n+1}(\t)},\ldots, z_{u_{n+m-1}(\t)}$.
The bound of the preceding display follows immediately.

Since $0\leq R_n\leq n$, we can then apply Kingman's subadditive ergodic theorem
to the sequence $(R_n)_{n\geq 1}$, and we get that $R_n/n$ converges almost surely. The fact
that the limit is constant is immediate from a simple zero-one law argument
(we could also verify that $\tau^*$ is ergodic). Furthermore, the limiting
constant $c_{\mu,\theta}$ is recovered by 
$$c_{\mu,\theta} = \lim_{n\to\infty} \frac{1}{n} \bE^*_{\mu,\theta}[R_n].$$
However, with the preceding notation,
\begin{align*}
\bE^*_{\mu,\theta}[R_n]&= \bE^*_{\mu,\theta}\Bigg[\sum_{i=0}^{n-1} \ind{\{z_{u_j} \not = z_{u_i},\,\forall
j\in \llbracket i+1,n-1\rrbracket\}}\Bigg]\\
&=\sum_{i=0}^{n-1} \bP^*_{\mu,\theta}(z_{u_j} \not = z_{u_i},\,\forall
j\in \llbracket i+1,n-1\rrbracket)\\
&=\sum_{i=0}^{n-1} \bP^*_{\mu,\theta}(z_{u_j}\not=0,\,\forall j\in \llbracket 1,n-i-1\rrbracket)
\end{align*}
using the shift invariance in the last equality. It now follows that
$$c_{\mu,\theta} = \lim_{n\to\infty} \frac{1}{n} \bE^*_{\mu,\theta}[R_n] =
\bP^*_{\mu,\theta}(z_{u_j}\not=0,\,\forall j\geq 1),$$
and the right-hand side is easily computed in the form given in the theorem, using the definition of 
$\bP^*_{\mu,\theta}$. \endproof

\medskip
Theorem~\ref{subaddi} does not give much information when the limiting 
constant $c_{\mu,\theta}$ is equal to $0$. In the next proposition, we give
sufficient conditions that ensure $c_{\mu,\theta}>0$. We let
$g_\mu$ denote the generating function of $\mu$,
$$g_\mu(r)\colonequals \sum_{k=0}^\infty \mu(k)\,r^k\;,\qquad 0\leq r\leq 1.$$
In the remaining part of this subsection, we assume that the random walk $S$ is transient (it is not
hard to see that $c_{\mu,\theta}=0$ if $S$ is recurrent). We denote the Green function of $S$ by $G_\theta$, that is 
$$G_\theta(x)\colonequals E_0\Big[\sum_{k=0}^\infty \ind{\{S_k=x\}}\Big]\;,\qquad x\in \Z^d.$$

\begin{proposition}
\label{suffcond}
{\rm(i)} The property $c_{\mu,\theta}>0$ holds if
$$\prod_{j=1}^\infty \Big( \frac{1 - g_\mu((1-G_\theta(S_j))_+)}{G_\theta(S_j)}\Big) >0\;,\qquad 
P_0\hbox{-a.s.}$$
\noindent{\rm(ii)} Suppose that the random walk $S$ is centered and has finite
moments of order $(d-1)\vee 2$. Then,
\begin{itemize}
\item if $\mu$ has finite variance, then $c_{\theta,\mu}>0$
if $d\geq 5$. 
\item if $\mu$ is in the domain of attraction of a stable distribution 
with index $\alpha\in(1,2)$, then $c_{\theta,\mu}>0$ if $d> \frac{2\alpha}{\alpha-1}$.
\end{itemize}
\end{proposition}

\proof (i) 
We have already noticed that $a_{\mu,\theta}>0$. We then observe that, for every $r\in [0,1)$,
\begin{equation}
\label{genera}
\sum_{k=0}^\infty \mu([k+1,\infty))\,r^k = \frac{1-g_\mu(r)}{1-r}.
\end{equation}
Next we can get a lower bound on the function $h_{\mu,\theta}(y)$ by
saying that the probability for tree-indexed random walk to visit the point 
$-y$ is bounded above by the expected value of the number of vertices at which
the random walk sits at $-y$. Since $\mu$ is critical, it follows that
$$h_{\mu,\theta}(y)\geq 1 - G_\theta(-y)$$
for every $y\in \Z^d$. Hence, for every $x\in\Z^d$,
$$\sum_{y\in\Z^d} \theta(y)\,h_{\mu,\theta}(x+y)
\geq 1 - \sum_{y\in\Z^d} \theta(y) G_\theta(-x-y).$$
However,
$$\sum_{y\in\Z^d} \theta(y) G_\theta(-x-y) = 
\sum_{y\in\Z^d} \theta(y)\, E_{x+y}\Big[\sum_{k=0}^\infty \ind{\{S_k=0\}}\Big]
=E_x\Big[\sum_{k=1}^\infty \ind{\{S_k=0\}}\Big]\leq G_\theta(-x).$$
Consequently, using~\eqref{genera}, we have, for all $x$ such that $G_\theta(-x)>0$,
$$\Phi_{\mu,\theta}(x)\geq \frac{1-g_\mu((1-G_\theta(-x))_+)}{G_\theta(-x)}.$$
The assertion in (i) follows, noting that $G_\theta(S_j)>0$ for every $j\geq 0$, $P_0$-a.s.

\smallskip
\noi (ii) If $S$ is centered with finite moments of order $(d-1)\vee 2$, then a standard bound for the 
Green function (see e.g.~\cite[Th\'eor\`eme 3.5]{LGRW}) gives the existence of a constant $C_\theta$ such that,
for every $x\in\Z^d$,
\begin{equation}
\label{Greenbound}
G_\theta(x)\leq C_\theta\,|x|^{2-d}
\end{equation}
(recall that we assume that $S$ is transient, so 
that necessarily $d\geq 3$ here).

Suppose first that $\mu$ has a finite variance $\sigma_{\mu}^2$. Then,
$$g_{\mu}(1-s) = 1- s + \frac{\sigma_{\mu}^2}{2}s^2 +o(s^2)$$
as $s\to 0$. Consequently,
$$\frac{1-g_\mu(1-s)}{s}= 1-\frac{\sigma_{\mu}^2}{2}s + o(s)$$
as $s\to 0$. By taking $s=G_\theta(S_j)$, we see that the condition in (i)
will be satisfied if
$$\sum_{j=1}^\infty G_\theta(S_j) < \infty\;,\qquad P_0\hbox{-a.s.}$$
However, using the local limit theorem and the preceding bound for $G_\theta$,
it is an easy matter to verify that the property
$$E_0\Big[\sum_{j=1}^\infty G_\theta(S_j) \Big]< \infty$$
holds if $d\geq 5$. 
This gives the desired result when $\mu$ has a finite variance.

Suppose now that $\mu$ is in the domain of attraction of 
a stable distribution with index $\alpha\in(1,2)$. Then the generating
function of $\mu$ must satisfy the property
$$g(1-s)= 1- s + s^\alpha\,L(s)$$
where $L$ is slowly varying as $s\downarrow 0$ (see e.g.~the discussion
in~\cite[p.60]{DLG}). By the same argument as above, we see that the
condition in (i) will be satisfied if
$$\sum_{j=1}^\infty G_\theta(S_j)^{\alpha -1} < \infty\;,\qquad P_0\hbox{-a.s.}$$
and this holds if 
$$(\alpha-1)(d-2)/2 >1,$$
which completes the proof. \endproof

\smallskip
\noi{\bf Remarks.} 1. The moment assumption in (ii) can be weakened 
a little: According to~\cite{Lawler}, the bound~\eqref{Greenbound} holds 
provided the random walk $S$ (is centered and) has moments of
order $(d-2+\ve)\vee 2$ for some $\ve>0$. However moments of order $d-2$
would not be sufficient for this bound.

\medskip

\noi 2. Suppose that the random walk $S$ satisfies the 
conditions in part (ii) of the proposition. If $\mu$ has finite variance, it is not hard to verify that
$c_{\mu,\theta}=0$ if $d\leq 4$. Let us briefly sketch the argument. It is enough to
consider the case $d=4$. Under 
the probability measure $\Pi^*_{\mu,\theta}$, write 
$N_x$ for the number of vertices whose spatial location is equal to $x$. Then, if $x\not = 0$,
$$\Pi^*_{\mu,\theta}[N_x]= G_\theta(x)\geq C'_\theta|x|^{-2}$$
for some constant $C'_\theta>0$. On the other hand,
standard arguments for Galton-Watson trees show that there exists a constant $K_\mu$ such that
$$\Pi^*_{\mu,\theta}[(N_x)^2] \leq K_\mu \sum_{z\in \Z^4} G_\theta(z)G_\theta(x-z)^2.$$
Using \eqref{Greenbound} and simple calculations, we obtain the existence 
of a constant $K'_{\mu,\theta}$ such that, for every 
$x\in \Z^4$ with $|x|\geq 2$,
$$\Pi^*_{\mu,\theta}[(N_x)^2] \leq K'_{\mu,\theta}\,|x|^{-2}\,\log|x|.$$
Hence, for every 
$x\in \Z^4$ with $|x|\geq 2$,
$$1-h_{\mu,\theta}(x)= \Pi^*_{\mu,\theta}(N_{-x}\geq 1)
\geq \frac{(\Pi^*_{\mu,\theta}[N_{-x})])^2}{\Pi^*_{\mu,\theta}[(N_{-x})^2] }\geq (C'_\theta)^2
(K'_{\mu,\theta})^{-1}\, |x|^{-2}\,(\log|x|)^{-1}.$$
The property $c_{\mu,\theta}=0$ now follows easily.
In the next section, we will see (in a particular case)
that the proper normalization factor for $R_n$ is $(\log n)/n$ when $d=4$.

\medskip
\noi 3.
It is an interesting question whether the condition $d>\frac{2\alpha}{\alpha-1}$
is also sharp when $\mu$ is in the domain of attraction of a stable distribution
of index $\alpha$. We will not discuss this problem here as our main
interest lies in the case when $\mu$ has finite variance. 

\subsection{Conditioned trees}

Our goal is now to obtain an analog of the convergence of Theorem \ref{subaddi}
for random walk indexed by a single Galton-Watson tree
conditioned to be large. 
Recall from subsection~\ref{fitree} 
the notation $\T_f^*$ for the set of all spatial trees. 
If $\t^*=(\t,(z_u)_{u\in\t})$ is a spatial tree with at least $n$ vertices, we keep the same notation 
$R_n(\t^*)$
for the number of distinct points in the sequence $z_{u_0},z_{u_1},\ldots,z_{u_{n-1}}$, where
$u_0,u_1,\ldots u_{\#\t-1}$ are the vertices of $\t$ listed in 
lexicographical order. Also recall
from subsection~\ref{fitree} the definition of the probability measure $\Pi^*_{\mu,\theta}$
on $\T^*_f$. 

\begin{proposition}
\label{rangecondi}
Assume that $\mu$ has finite variance $\sigma_{\mu}^2$, or that 
$\mu$ is in the domain of attraction of a stable distribution with index $\alpha\in(1,2)$.
For every $n\geq 1$, let $\ts^*_{(> n)}$ be a random spatial tree distributed according to the probability measure
$\Pi^*_{\mu,\theta}(\cdot\mid \#\t> n)$. Then, for every $a\in(0,1]$,
$$\frac{1}{n}\,R_{\lfloor an\rfloor}(\ts^*_{(> n)}) \build{\la}_{n\to\infty}^{} c_{\mu,\theta}\,a$$
in probability.
\end{proposition}

\proof We first consider the case when $\mu$ has finite variance $\sigma_{\mu}^2$. Let $\ts^*=(\ts,(\mathscr{Z}_u)_{u\in\ts})$ be a $\T^*$-valued random variable distributed according to 
$\bP^*_{\mu,\theta}$ under the probability measure $P$. Recall the notation $\ts_j$, for $j\in \Z_-$, introduced in 
subsection~\ref{infitree}. By construction, the ``subtrees''
$$\ts_0,[\ts_{-1}]_1,[\ts_{-1}]_{2},\ldots,[\ts_{-1}]_{k_\varnothing(\ts_{-1})},[\ts_{-2}]_1,[\ts_{-2}]_{2},\ldots,[\ts_{-2}]_{k_\varnothing(\ts_{-2})},
[\ts_{-3}]_1,\ldots$$
then form an infinite sequence of independent random trees distributed
according to $\Pi_\mu$. To simplify notation we denote this sequence by
$\ts_{(0)},\ts_{(1)},\ts_{(2)},\ldots$. We then introduce the height process 
$(H_k)_{k\geq 0}$ associated with this sequence of trees
(see~\cite[Section 1]{LG1}). This means that, for every
$j\geq 0$, we first enumerate the vertices of $\ts_{(j)}$ in lexicographical order, then we concatenate the finite sequences obtained in this way to get an
infinite sequence $(v_k)_{k\geq 0}$ of elements in~$\mathcal{U}$, and
we finally set $H_k \colonequals |v_k|$ for every $k\geq 0$. Note that the infinite sequence of vertices $(v_k)_{k\geq 0}$ thus obtained is essentially the same as the sequence $(u_{k}(\ts))_{k\geq 0}$ introduced in subsection~\ref{sec:random-walk-indexed}.

Then (see e.g.~\cite[Theorem 1.8]{LG1}), we have the convergence in distribution
\begin{equation}
\label{heightforest}
\Big(\frac{1}{\sqrt{n}}\,H_{\lfloor nt \rfloor}\Big)_{t\geq 0} \build{\longrightarrow}
_{n\to\infty}^{\rm(d)} \Big(\frac{2}{\sigma_{\mu}}\, |\beta_t|\Big)_{t\geq 0}
\end{equation}
where $(\beta_t)_{t\geq 0}$ denotes a standard linear Brownian motion. Next, for every integer $n\geq 1$, set
$$k_n:=\inf\big\{k\geq 0: \#\ts_{(k)} > n\big\}.$$
Clearly, the tree $\ts_{(k_n)}$ is distributed according to
$\Pi_\mu(\cdot\mid \# \t> n)$. Also set
$$d_{k_n}:= \sum_{0\leq j<k_n} \#\ts_{(j)}.$$
Using the convergence~\eqref{heightforest}, it is not hard to prove (see e.g.~the proof of Theorem 5.1 in~\cite{LG2}) that
\begin{equation}
\label{convdebut}
\frac{1}{n}\,d_{k_n} \build{\longrightarrow}
_{n\to\infty}^{\rm(d)} D_1
\end{equation}
where $D_1$ denotes the initial time of the first excursion of $\beta$
away from $0$ with duration greater than $1$. 

By Theorem~\ref{subaddi} and an obvious monotonicity argument, we have
for every integer $K>0$,
$$\lim_{n\to\infty} \sup_{0\leq t\leq K} \Big| \frac{1}{n} R_{\lfloor nt\rfloor}
(\ts^*)- c_{\mu,\theta}\,t\Big| = 0\;,\qquad\hbox{a.s.}$$
and it follows that
$$\lim_{n\to\infty}
\Big| \frac{1}{n} R_{(d_{k_n}+\lfloor an\rfloor)\wedge Kn}(\ts^*)
-\frac{1}{n} R_{d_{k_n}\wedge Kn}(\ts^*)
- c_{\mu,\theta} \Big( (\frac{d_{k_n}}{n} + a)\wedge K - \frac{d_{k_n}}{n}\wedge K\Big)\Big| = 0\;,\qquad\hbox{a.s.}$$
Since $K$ can be chosen arbitrarily large, we deduce from the last convergence
and~\eqref{convdebut} that we have 
$$\lim_{n\to\infty} \frac{1}{n}
\Big( R_{d_{k_n}+\lfloor an\rfloor}(\ts^*) - R_{d_{k_n}}(\ts^*)\Big) = c_{\mu,\theta}\,a$$
in probability. 

Let $\ts^*_{(k_n)}$ stand for the spatial tree obtained from 
$\ts_{(k_n)}$ by keeping the spatial positions induced by $\ts^*$. Then,
by
construction, we have 
$$R_{\lfloor an\rfloor}(\ts^*_{(k_n)})\geq R_{d_{k_n}+\lfloor an\rfloor}(\ts^*) - R_{d_{k_n}}(\ts^*).$$
Therefore, using the preceding convergence in probability, we obtain
that, for every fixed $\ve>0$,
\begin{equation}
\label{condirangetech0}
P\Big( R_{\lfloor an\rfloor}(\ts^*_{(k_n)}) \geq (c_{\mu,\theta}a -\ve)n\Big)
\build{\longrightarrow}_{n\to\infty}^{} 1.
\end{equation}

We claim that we have also
\begin{equation}
\label{condirangetech}
P\Big( R_{\lfloor an\rfloor}(\ts^*_{(k_n)}) \leq (c_{\mu,\theta}a +\ve)n\Big)
\build{\longrightarrow}_{n\to\infty}^{} 1.
\end{equation}
To see this, we argue by contradiction and suppose that for all $n$ belonging 
to a sequence $(n_j)_{j\geq 1}$ converging to infinity, we have
$$P\Big( R_{\lfloor an\rfloor}(\ts^*_{(k_n)}) > (c_{\mu,\theta}a +\ve)n\Big)\geq \delta$$
for some $\delta>0$ independent of $n$. 
We suppose that $c_{\mu,\theta}>0$ (the case when 
$c_{\mu,\theta}=0$ is easier). We observe that, for every
fixed $n$, the tree $\ts_{(k_n)}$ and the quantity $R_{\lfloor an\rfloor}(\ts^*_{(k_n)})$ are
independent of the random variable $d_{k_n}$. Notice that $\ts^*_{(k_n)}$
is {\em not} independent of $d_{k_n}$, because the value of $d_{k_n}$
clearly influences the distribution of the spatial location of the root
of $\ts_{(k_n)}$. However, if we simultaneously translate all spatial locations
of $\ts^*_{(k_n)}$ so that the new location of the root is $0$, the new locations
become independent of $d_{k_n}$, and the translation does not
affect $R_{\lfloor an\rfloor}(\ts^*_{(kn)})$. On the other hand, from the convergence 
in distribution~\eqref{convdebut}, we can find $\delta'>0$ such that,
for every sufficiently large $n$,
$$P\Big(d_{k_n}\leq \frac{\ve}{2 c_{\mu,\theta}}\, n\Big) \geq \delta'.$$
Using the preceding independence property, we conclude that,
for every sufficiently large $n$ in the sequence $(n_j)_{j\geq 1}$,
$$P\Big(R_{\lfloor(\ve n/2c_{\mu,\theta})+ an\rfloor}(\ts^*)\geq (c_{\mu,\theta}a +\ve)n\Big)
\geq P\Big(d_{k_n}\leq \frac{\ve}{2 c_{\mu,\theta}}\, n\Big)\,P\Big( R_{\lfloor an\rfloor}(\ts^*_{(k_n)}) > (c_{\mu,\theta} a+\ve)n\Big)
\geq \delta \delta'.$$
However Theorem~\ref{subaddi} implies that
$$\frac{1}{n} R_{\lfloor(\ve n/2c_{\mu,\theta})+an\rfloor}(\ts^*) \build{\la}_{n\to\infty}^{} c_{\mu,\theta}\,a + \frac{\ve}{2}\;,\qquad\hbox{a.s.}$$
and so we arrive at a contradiction, which completes the proof of~\eqref{condirangetech}. 

By construction, the tree $\ts_{(k_n)}$ is distributed according to $\Pi_{\mu}(\cdot\mid \#\t > n)$,
and if we shift all spatial locations of $\ts^*_{(k_n)}$ so that the new location of the root is $0$,
we get a random spatial tree distributed according to $\Pi^*_{\mu,\theta}(\cdot\mid \#\t> n)$.
The convergence of the proposition thus follows from~\eqref{condirangetech0} and~\eqref{condirangetech}. 

The proof in the case when $\mu$ is in the domain of attraction of a stable distribution with index $\alpha\in(1,2)$ is essentially the same, noting that Theorems 2.3.1 and 2.3.2 in \cite{DLG} give
an analog of the convergence (\ref{heightforest}), where the role of reflected Brownian motion
is played by the so-called height process associated with the stable L\'evy process with
index $\alpha$. We omit the details. 
\endproof

\medskip
We now would like to get a statement analogous to Proposition~\ref{rangecondi} 
for a tree conditioned to have  a fixed number of vertices. This will follow
from Proposition~\ref{rangecondi} by
an absolute continuity argument. Before stating the result, we need to introduce 
some notation. Let $\mathcal{G}$ be the smallest subgroup of $\Z$ that contains the
support of $\mu$. Plainly, the cardinality of the vertex set of a tree distributed according
to $\Pi_\mu$ belongs to $1+\mathcal{G}$. On the other hand, for every sufficiently large 
integer $p\in 1+\mathcal{G}$, we have $\Pi_\mu(\#\t=p)>0$, so that the definition 
of $\Pi_\mu(\cdot \mid \#\t=p)$ makes sense. 

If $\t^*=(\t,(z_u)_{u\in\t})$ is a spatial tree, we write $\mathcal{R}(\t^*)$
for the number of distinct elements in $\{z_u:u\in\t\}$.

\begin{theorem}
\label{mainsuper}
Assume that $\mu$ has finite variance $\sigma_{\mu}^2$, or that 
$\mu$ is in the domain of attraction of a stable distribution with index $\alpha\in(1,2)$.
For every sufficiently large integer $n\in \mathcal{G}$, let $\ts^*_{(n)}$ be a random spatial tree distributed according to the probability measure
$\Pi^*_{\mu,\theta}(\cdot\mid \#\t= n+1)$. Then,
$$\frac{1}{n}\,\mathcal{R}(\ts^*_{(n)})\; \build{\la}_{n\to\infty,\;n\in \mathcal{G}}^{} \;c_{\mu,\theta}$$
in probability.
\end{theorem}

\proof We assume in the proof that $\mathcal{G}=\Z$. Only minor modifications are needed to deal
with the general case. 

We first consider the case when $\mu$ has finite variance $\sigma_{\mu}^2$.
The arguments needed to derive Theorem~\ref{mainsuper} from Proposition~\ref{rangecondi}
are then similar to the proof of Theorem 6.1 in~\cite{LG2}. The basic idea is as follows. For every $a\in(0,1)$, the law under $\Pi_\mu(\cdot\mid \#\t=n+1)$ of the 
subtree obtained by keeping only the first $\lfloor an\rfloor$ vertices of $\t$ is absolutely continuous with respect to the law
under $\Pi_\mu(\cdot\mid \#\t>n)$ of the same subtree, with a density that is bounded independently of $n$. 
A similar property holds for spatial trees, and so we can use the convergence of Proposition \ref{rangecondi}, for a tree distributed
according to $\Pi^*_{\mu,\theta}(\cdot\mid \#\t>n)$, to get a similar convergence for a tree distributed according
to $\Pi^*_{\mu,\theta}(\cdot\mid \#\t=n+1)$. Let us give some
details for the sake of completeness. 

As previously, we write $u_0(\t),u_1(\t),\ldots,u_{\#\t-1}(\t)$ for the vertices of a plane tree $\t$
listed in lexicographical order. The Lukasiewisz path of $\t$ is then the finite sequence
$(X_\ell(\t),0\leq \ell\leq \#\t)$, which is defined inductively by
$$X_0(\t)=0\;,\quad X_{\ell+1}(\t)-X_\ell(\t)= k_{u_{\ell}(\t)}(\t)-1\;,\quad\hbox{for every }0\leq \ell <\#\t$$
where we recall that, for every $u\in\t$, $k_u(\t)$ is the number of children of $u$ in $\t$. 
The tree $\t$ is determined by its Lukasiewisz path.
A key result (see e.g.~\cite[Section 1]{LG1}) states that under $\Pi_\mu(\mathrm{d}\t)$, the Lukasiewisz path is 
distributed as a random walk on $\Z$ with jump distribution $\nu$
determined by $\nu(j)=\mu(j+1)$ for every $j\geq -1$, which starts from $0$ and is stopped at the first time when it hits $-1$
(in particular, the law of $\#\t$ under $\Pi_\mu(\mathrm{d}\t)$ coincides with the law
of the latter hitting time).
For notational convenience, we let $(Y_k)_{k\geq 0}$ be a random walk on $\Z$ with jump 
distribution $\nu$, which starts from $j$ under the probability measure $P_{(j)}$, and we set
$$T \colonequals \inf\{k\geq 0: Y_k=-1\}.$$

Next take $n$  large enough so that $\Pi_\mu(\#\t=n+1)>0$.  Fix $a\in(0,1)$, and 
consider a tree $\t$ such that $\#\t>n$. Then, the collection of
vertices $u_0(\t),\ldots, u_{\lfloor an\rfloor}(\t)$ forms a subtree of~$\t$ (because in the 
lexicographical order the parent of a vertex comes before this vertex), and
we denote this subtree by $\rho_{\lfloor an\rfloor}(\t)$. It is elementary
to verify that $\rho_{\lfloor an\rfloor}(\t)$ is determined by the sequence
$(X_\ell(\t),0\leq \ell\leq \lfloor an\rfloor)$. Let $f$ be a bounded function
on $\Z^{\lfloor an\rfloor+1}$. Using the Markov property at time $\lfloor an\rfloor$
for the random walk with jump distribution $\nu$, one
verifies that
\begin{align}
\label{supertech1}
&\Pi_\mu \Big[ f((X_k)_{0\leq k\leq \lfloor an\rfloor})\,\Big| \,\#\t=n+1\Big]\notag\\
&\qquad=\frac{P_{(0)}(T>n)}{P_{(0)}(T=n+1)}\, \Pi_\mu \Bigg[ f((X_k)_{0\leq k\leq \lfloor an\rfloor})
\frac{\psi_n(X_{\lfloor an\rfloor})}{\psi'_n(X_{\lfloor an\rfloor})}\,
\,\Bigg| \,\#\t>n\Bigg]
\end{align}
where, for every integer $j\geq 0$,
$$\psi_n(j)= P_{(j)}(T=n+1-\lfloor an\rfloor)\;,\ \psi'_n(j) = P_{(j)}(T>n-\lfloor an\rfloor).$$
See~\cite[pp.742-743]{LG2} for details of the derivation of~\eqref{supertech1}. 
We now let $n$ tend to infinity. Using Kemperman's formula 
(see e.g.~Pitman~\cite[p.122]{pitman}) and a standard local limit
theorem, one easily checks that, for every $c>0$,
\begin{equation}
\label{supertech2}
\lim_{n\to \infty} \Bigg(\sup_{j\geq c\sqrt{n}} \Bigg| \frac{P_{(0)}(T>n)}{P_{(0)}(T=n+1)}\,
\frac{\psi_n(j)}{\psi'_n(j)} - \Gamma_a\Big(\frac{j}{\sigma_{\mu}\sqrt{n}}\Big)\Bigg|\Bigg) =0,
\end{equation}
where for every $x\geq 0$, 
$$\Gamma_a(x)= \frac{2(2\pi(1-a)^3)^{-1/2}\,\exp(-x^2/2(1-a))}{\int_{1-a}^\infty \mathrm{d}s\,
(2\pi s^3)^{-1/2}\,\exp(-x^2/2s)}.$$
See again~\cite[pp.742-743]{LG2} for details. Note that the function $\Gamma_a$
is bounded over $\R_+$. Furthermore, from the local limit theorem again, it is easy to
verify that
\begin{equation}
\label{supertech3}
\lim_{c\downarrow 0}\limsup_{n\to\infty} \Pi_\mu(X_{\lfloor an\rfloor}\leq c\sqrt{n}\,|\, \#\t=n+1)=0\,,\;
 \lim_{c\downarrow 0}\limsup_{n\to\infty} \Pi_\mu(X_{\lfloor an\rfloor}\leq c\sqrt{n}\,|\, \#\t>n)=0.
\end{equation}
(We take this opportunity to point out that the analogous statement in~\cite[p.743]{LG2} is written
incorrectly.) By combining~\eqref{supertech1}, \eqref{supertech2} and~\eqref{supertech3},
we obtain that, for any uniformly bounded sequence of functions $(f_n)_{n\geq 1}$
on $\Z^{\lfloor an\rfloor+1}$, we have
\begin{equation}
\label{supertech4}
\lim_{n\to\infty} \Big| \Pi_\mu \Big[ f_n((X_k)_{0\leq k\leq \lfloor an\rfloor})\,\Big| \,\#\t=n+1\Big]
- \Pi_\mu \Big[ f_n((X_k)_{0\leq k\leq \lfloor an\rfloor})\,
\Gamma_a(\frac{X_{\lfloor an\rfloor}}{\sigma_{\mu}\sqrt{n}})\,\Big| \,\#\t>n\Big]\Big|=0.
\end{equation}
This convergence applies in particular to the case when, for every $n$, $f_n((X_k)_{0\leq k\leq \lfloor an\rfloor})$
is a function of the tree $\rho_{\lfloor an\rfloor}(\t)$. If we now replace $\Pi_\mu$
by $\Pi^*_{\mu,\theta}$, the same convergence still holds, and we can even allow the function 
of the tree $\rho_{\lfloor an\rfloor}(\t)$ to depend also on the spatial locations of the
vertices of $\rho_{\lfloor an\rfloor}(\t)$ (the point is that the conditional distribution of these
spatial locations given the tree $\t$ only depends on the subtree $\rho_{\lfloor an\rfloor}(\t)$). 
Consequently, if $\ve >0$ is fixed, we have
$$\lim_{n\to\infty}\Big| \Pi^*_{\mu,\theta} \Big[\ind{\{ |R_{\lfloor an\rfloor}- c_{\mu,\theta}an|>\ve n\}}\,\Big| \,\#\t=n+1\Big]
- \Pi^*_{\mu,\theta} \Big[\ind{\{ |R_{\lfloor an\rfloor}- c_{\mu,\theta}an|>\ve n\}}
\,\Gamma_a(\frac{X_{\lfloor an\rfloor}}{\sigma_{\mu}\sqrt{n}})\,\Big| \,\#\t>n\Big]\Big| =0.$$
Recalling that the function $\Gamma_a$ is bounded, and using Proposition~\ref{rangecondi},
we now obtain that 
$$\lim_{n\to\infty}  \Pi^*_{\mu,\theta} \Big( |R_{\lfloor an\rfloor}- c_{\mu,\theta}an|>\ve n\,\Big| \,\#\t=n+1\Big)
=0.$$
Since $0\leq \mathcal{R}(\t^{*}) - R_{\lfloor an\rfloor}(\t^{*})\leq n+1 - \lfloor an\rfloor$, $\Pi^*_{\mu,\theta}(\cdot\mid \#\t=n+1)$-a.s., and $a$ can be chosen arbitrarily close
to $1$, the convergence in Theorem~\ref{mainsuper} follows. 

Very similar arguments can be used in the case when $\mu$ is in the domain 
of attraction of a stable distribution with index $\alpha\in(1,2)$. We now refer 
to the proof of Lemma 3.3 in \cite{Kor} for the exact analogs of the properties 
(\ref{supertech1}) -- (\ref{supertech4}) used in the finite variance case. 
We leave the details to the reader. 
\endproof

\smallskip
The case $d\geq 5$ of Theorem \ref{treeSRW} follows from Theorem \ref{mainsuper} and 
Proposition \ref{suffcond}, noting that when $\mu$ is the critical geometric distribution,
a tree distributed according to $\Pi_\mu(\cdot\mid\#\t=n)$ is uniformly distributed over
the set of all plane trees with $n$ vertices (see e.g. \cite[Section 1.5]{LG1}).

\section{The critical dimension}

In this section, we discuss the dimension $d=4$, which is critical in the case of random walks
that are centered and have sufficiently high moments. We restrict our attention 
to the case when the offspring distribution is geometric with parameter $1/2$.
Our main tool is the discrete snake, which is a path-valued Markov chain 
that can be used to generate the spatial positions of the tree-indexed random walk. 

\subsection{Limit theorems}

We now let $\theta$ be a symmetric probability distribution on
$\Z^4$. We assume that $\theta$ has small exponential moments and 
is not supported on a strict subgroup
of $\Z^4$.  As previously, we write  $S=(S_k)_{k\geq 0}$ for the 
random walk in $\Z^4$ with jump distribution $\theta$, and we now assume that $S$ starts from $0$ under the
probability measure $P$.  We will also assume for simplicity that 
the covariance matrix $M_\theta$ of~$\theta$ is 
of the form $\sigma^2\,{\rm Id}$, where ${\rm Id}$ is the four-dimensional identity
matrix and $\sigma >0$. This isotropy condition can be removed,
and the reader will easily check that all subsequent arguments remain
valid for a non-isotropic random walk: the role of $\sigma^2$ is then played
by $({\rm det}\,M_\theta)^{1/4}$.

We first introduce the free discrete snake associated with $\theta$. This is
a Markov chain with values in the space $\mathcal{W}$ that we now define. The space
$\mathcal{W}$ is the set of all semi-infinite discrete paths $w=(w(k))_{k\in\rrbracket -\infty,
\zeta\rrbracket}$ with values in $\Z^4$. Here $\zeta= \zeta(w)\in\Z$
is called the lifetime of $w$. We often write $\wh w = w(\zeta(w))$ for the endpoint of $w$.

If $w\in \mathcal{W}$, we let $\overline w$ stand for the new path 
obtained by ``erasing'' the endpoint of $w$, namely
$\zeta(\overline w)= \zeta(w)-1$ and 
$\overline w(k)=w(k)$ for every $k\in\rrbracket-\infty,\zeta(w)-1\rrbracket$. If $x\in \Z^4$, 
we let $w\oplus x$ be the path obtained from $w$ by ``adding'' the
point $x$ to $w$, namely $\zeta(w\oplus x)=\zeta(w)+1$,
$(w\oplus x)(k)=w(k)$ for every $k\in \rrbracket -\infty,
\zeta(w)\rrbracket$ and $(w\oplus x)(\zeta(w)+1)=x$. 

The free discrete snake is the Markov chain
$(W_n)_{n\geq 0}$ in $\mathcal{W}$ whose transition
kernel is defined by
$$Q(w,\mathrm{d}w')=\frac{1}{2}\,\delta_{\overline w}(\mathrm{d}w')
+ \frac{1}{2}\sum_{x\in\Z^4} \theta(x)\,\delta_{w\oplus(
\wh w+x)}(\mathrm{d}w').$$
We will write $\zeta_n=\zeta(W_n)$ to simplify notation. It will also be convenient 
to write $W_n^*$ for the path $W_n$ shifted so that its endpoint is $0$: 
$W_n^*(k) = W_n(k)-\wh W_n$ for every $k\in\rrbracket -\infty,\zeta_n\rrbracket$.

\begin{figure}[!htbp]
 \begin{center}
 \includegraphics[width=11cm]{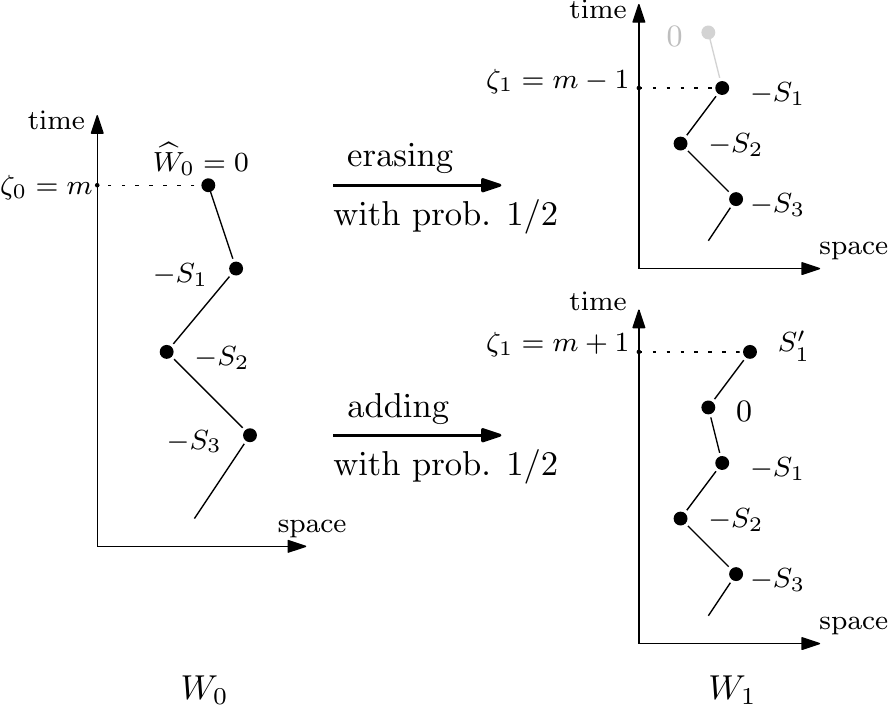}
 \caption{\textsf{The discrete snake under $\mathbb{P}_m$. In this illustration, $S'_{1}$ is an independent copy of~$S_{1}$.} \label{fig-snake}}
 \end{center}
\end{figure}

If $w\in \mathcal{W}$, $\P_{(w)}$ will denote the probability measure 
under which the discrete snake $W$ starts from $w$.
For every integer $m\in\Z$, we also write $\P_m$ for a probability measure under which $\zeta_0=m$ a.s.~and  the initial value
$W_0$ of the discrete snake is distributed as $(-S_{m-k})_{k\in\rrbracket-\infty,m\rrbracket}$ (since 
$S$ is symmetric we could omit the minus sign here).  We write $\P$ for $\P_0$. As usual, the expectation under $\P_m$, resp.~under $\P$, is denoted by $\E_m$, resp.~by~$\E$. Note that 
$(\zeta_n)_{n\geq 0}$ is a simple random walk on $\Z$ started from $m$
under $\P_m$. We will use the notation
$$\tau_p \colonequals \inf\{n\geq 0: \zeta_n=\zeta_0-p\}$$
for every integer $p\geq 0$.

Furthermore, from the form of the transition kernel of
the discrete snake, it is easy to verify that for every $n\geq0$, for 
every integer $\ell\in\Z$ such that $\P_m(\zeta_n=\ell)>0$, the conditional
distribution of $W_n^*$
 under $\P_m(\,\cdot\mid \zeta_n=\ell)$, coincides with the distribution of $W_0$
under $\P_\ell$.

\begin{proposition}
\label{firstesti}
We have
$$\lim_{n\to\infty} (\log n)\,\P(\wh W_k\not =0,\forall k\in\llbracket 1,n\rrbracket)=4\pi^2 \sigma^4.$$
Furthermore,
$$\lim_{p\to\infty} (\log p)\,\P(\wh W_k\not =0,\forall k\in\llbracket 1,\tau_p\rrbracket)=2\pi^2 \sigma^4.$$
\end{proposition}

The proof of Proposition~\ref{firstesti} is given in subsection~\ref{mainest}
below. 
Our first theorem is concerned with the range of the free snake.

\begin{theorem}
\label{freesnake4D}
Set $R_n \colonequals \#\big\{\wh W_0,\wh W_1,\ldots,\wh W_n\big\}$ for every integer $n\geq 0$.
We have
$$\frac{\log n}{n} \,R_n \build{\la}_{n\to\infty}^{L^2(\P)} 4\pi^2\sigma^4.$$
\end{theorem}

\proof We first observe that
$$
 \E[ R_n]
= \E\Big[ \sum_{i=0}^n \ind{\{\wh W_j \not = \wh W_i\;,
\forall j\in \llbracket i+1,n\rrbracket\}}\Big]
=\sum_{i=0}^n \P\big(\wh W_j \not = \wh W_i\;,
\forall j\in \llbracket i+1,n\rrbracket\big).
$$
Then, by applying the Markov property of the free snake, we have
\begin{align*}
\E[ R_n]
&= \sum_{i=0}^n
\E\Big[ \P_{(W_i)}(\wh W_j \not = \wh W_0\;,
\forall j\in \llbracket 1,n-i\rrbracket)\Big]\\
&= \sum_{i=0}^n
\E\Big[ \P_{(W_i^*)}(\wh W_j \not = \wh W_0\;,
\forall j\in \llbracket 1,n-i\rrbracket)\Big]\\
&= \sum_{i=0}^n \P(\wh W_j\not = 0, \forall j\in\llbracket 1,n-i\rrbracket),
\end{align*}
where the second equality is easy by translation invariance, and the last one is a simple consequence of the remark 
before the statement of Proposition~\ref{firstesti}. Using now the
result of Proposition~\ref{firstesti}, we get
\begin{equation}
\label{moment1freesnake}
\lim_{n\to\infty} \frac{\log n}{n}\, \E[ R_n]= 4\pi^2 \sigma^4.
\end{equation}

Let us turn to the second moment. We have similarly
\begin{align*}
\E\big[(R_n) ^2\big]
&=\E\Big[ \sum_{i=0}^n\sum_{j=0}^n \ind
{\{\wh W_k \not = \wh W_i\;,
\forall k\in \llbracket i+1,n\rrbracket; \wh W_\ell \not = \wh W_j\;,
\forall \ell\in \llbracket j+1,n\rrbracket\}}\Big]\\
&= 2\sum_{0\leq i<j\leq n} \P\Big(\wh W_k \not = \wh W_i\;,
\forall k\in \llbracket i+1,n\rrbracket; \wh W_\ell \not = \wh W_j\;,
\forall \ell\in \llbracket j+1,n\rrbracket\Big) +\E[R_n]\\
&= 2\sum_{0\leq i<j\leq n} 
\E\Big[ \P_{(W_i)}(\wh W_k \not = \wh W_0\;,
\forall k\in \llbracket 1,n-i\rrbracket; \wh W_\ell \not = \wh W_{j-i}\;,
\forall \ell\in \llbracket j-i+1,n-i\rrbracket)\Big]\\
&\qquad + \E[R_n]\\
&=2\sum_{0\leq i<j\leq n} 
\P\Big(\wh W_k \not = 0\;,
\forall k\in \llbracket 1,n-i\rrbracket; \wh W_\ell \not = \wh W_{j-i}\;,
\forall \ell\in \llbracket j-i+1,n-i\rrbracket\Big) + \E[R_n],
\end{align*}
where the last equality again follows from the observation preceding
Proposition~\ref{firstesti}. Let us fix $\alpha\in (0,1/4)$ and define
$$\sigma_n:=\inf\big\{k\geq 0: \zeta_k \leq -n^{\frac{1}{2}-\alpha}\big\}.$$
By standard estimates, we have
$$\lim_{n\to \infty} (\log n)^2\, \P\big(\sigma_n\leq n^{1-3\alpha}
\hbox{ or } \sigma_n \geq n^{1-\alpha}\big) =0.$$
Thus, using also~\eqref{moment1freesnake},
\begin{align*}
&\limsup_{n\to\infty} \Big(\frac{\log n}{n}\Big)^2\,
\E[(R_n)^2]
=\limsup_{n\to\infty} 2\Big(\frac{\log n}{n}\Big)^2
\sum_{0\leq i<j\leq n} 
\P\Big(\wh W_k \not = 0\;,
\forall k\in \llbracket 1,n-i\rrbracket; \\
&\hspace{6cm} \wh W_\ell \not = \wh W_{j-i}\;,
\forall \ell\in \llbracket j-i+1,n-i\rrbracket; 
n^{1-3\alpha}\leq \sigma_n\leq n^{1-\alpha}\Big).
\end{align*}
Clearly, in order to study the limsup in the right-hand side, we may restrict the
sum to indices $i$ and $j$ such that  $j-i> n^{1-\alpha}$. However,
 if $0\leq i<j\leq n$ are fixed such that $j-i> n^{1-\alpha}$,
\begin{align*}
&\P\Big(\wh W_k \not = 0\;,
\forall k\in \llbracket 1,n-i\rrbracket; \wh W_\ell \not = \wh W_{j-i}\;,
\forall \ell\in \llbracket j-i+1,n-i\rrbracket; 
n^{1-3\alpha}\leq \sigma_n\leq n^{1-\alpha}\Big)\\
&\leq \P\Big(\wh W_k \not = 0\;,
\forall k\in \llbracket 1,\sigma_n\rrbracket; \wh W_\ell \not = \wh W_{j-i}\;,
\forall \ell\in \llbracket j-i+1,n-i\rrbracket; 
n^{1-3\alpha}\leq \sigma_n\leq n^{1-\alpha}\Big)\\
&= \P\Big(\wh W_k \not = 0\;,
\forall k\in \llbracket 1,\sigma_n\rrbracket;n^{1-3\alpha}\leq \sigma_n\leq n^{1-\alpha}\Big)\,\P\Big(\wh W_\ell \not =0\;,\forall \ell \in\llbracket 1,n-j\rrbracket\Big).
\end{align*}
To derive the last equality, we use the strong Markov property at time $\sigma_n$
and then, after conditioning on $\sigma_n=m$, the Markov property
at time $j-i-m$ for
the free snake shifted at time $\sigma_n$ and the observation preceding
Proposition~\ref{firstesti}. Now obviously,
$$\P\Big(\wh W_k \not = 0\;,
\forall k\in \llbracket 1,\sigma_n\rrbracket;n^{1-3\alpha}\leq \sigma_n\leq n^{1-\alpha}\Big)\leq \P\Big(\wh W_k \not = 0\;,
\forall k\in \llbracket 1,\lfloor n^{1-3\alpha}\rfloor\rrbracket\Big),$$
and it follows that
\begin{align*}
&\limsup_{n\to\infty} \Big(\frac{\log n}{n}\Big)^2\,
\E\big[(R_n)^2\big]\\
&\leq \limsup_{n\to\infty} 2\Big(\frac{\log n}{n}\Big)^2
\build{\sum_{0\leq i<j\leq n}}_{j-i>n^{1-\alpha}}^{}  
 \P\Big(\wh W_k \not = 0\;,
\forall k\in \llbracket 1,\lfloor n^{1-3\alpha}\rfloor\rrbracket\Big)\,
\P\Big(\wh W_\ell \not =0\;,\forall \ell \in\llbracket 1,n-j\rrbracket\Big)\\
&= \frac{1}{1-3\alpha} (4\pi^2 \sigma^4)^2
\end{align*}
by Proposition~\ref{firstesti}. Since $\alpha$ can be chosen arbitrarily small,
we get
\begin{equation}
\label{moment2freesnake}
\limsup_{n\to\infty} \Big(\frac{\log n}{n}\Big)^2\,
\E\big[(R_n)^2\big]\leq (4\pi^2 \sigma^4)^2.
\end{equation}
Theorem~\ref{freesnake4D} is an immediate consequence of~\eqref{moment1freesnake} and~\eqref{moment2freesnake}. \endproof

\medskip
We now aim to prove a result similar to Theorem~\ref{freesnake4D} for 
the ``excursion'' of the discrete snake. We set
$$T \colonequals \inf\{k\geq 0: \zeta_k=-1\}.$$
For every integer $n\geq 1$, we let $W^{(n)}=(W^{(n)}_k)_{0\leq k\leq 2n}$ 
be a process defined under $\P$, whose distribution 
coincides with the conditional distribution of $(W_k)_{0\leq k\leq 2n}$ knowing that 
 $T=2n+1$. To simplify
notation, we write $\zeta^{(n)}_k=\zeta(W^{(n)}_k)$. Note that
$(\zeta^{(n)}_k)_{0\leq k\leq 2n}$ is the contour function, also called depth-first walk, of a
Galton-Watson tree with geometric offspring distribution of parameter $1/2$,
conditioned to have $n+1$ vertices (see e.g.\,\cite[Chapter 6]{pitman}).
We have already noticed that the latter tree is uniformly distributed over plane trees with 
$n+1$ vertices.
 From the form 
of the transition mechanism of the discrete snake, it then follows that $\{\wh W^{(n)}_k\,,\,0\leq k\leq 2n\}$
is distributed as the set of all spatial locations of a random walk with jump distribution 
$\theta$ indexed by a uniform  random plane tree with $n+1$ vertices. 

We will need two simple estimates that we gather in the next lemma.

\begin{lemma}
\label{estimate-excu}
{\rm(i)} Let $r\geq 1$ be an integer. There exists a constant $C(r)$
such that, for every integers $n\geq 1$ and $m\geq 0$,
$$\E\left[(\#\{k\in\llbracket 0,2n\rrbracket : \zeta^{(n)}_k =m\})^r\right]
\leq C(r)\,(m+1)^r.$$
{\rm(ii)} Let  $\ve>0$. Then, for every $r>0$,
$$\P\Big(\sup_{0\leq k\leq 2n} \zeta^{(n)}_k > n^{\frac{1}{2}+\ve}\Big)
=O(n^{-r})$$
as $n\to \infty$. 
\end{lemma}

Part (i) of the lemma can be deduced from Theorem~1.13 in 
Janson \cite{Janson} using the connection between $\zeta^{(n)}$
and the critical geometric Galton-Watson tree (it is also possible to
give a direct argument), while Part (ii) is standard. Notice that Part (i) of Lemma~\ref{estimate-excu} implies 
$$\E\left[(\#\{k\in\llbracket 0,2n\rrbracket : \zeta^{(n)}_k \leq n^{\frac{1}{2}-\frac{\alpha}{2}}\})^2\right]=o\Big(\big(\frac{n}{\log n}\big)^2\Big)$$
as $n\to \infty$. 

We will make a repeated use of Kemperman's formula
for simple random walk (see \cite[p.122]{pitman} for a more general
version): For every
choice of the integers $m,k$ such that $k>m\geq 0$,
\begin{equation}
\label{Kemp}
\P_m(T=k)= \frac{m+1}{k}\,\P_m(\zeta_k=-1)=\frac{m+1}{k}\,\P_0(\zeta_k=m+1).
\end{equation}
Together with this formula, we will use the local limit theorem for simple random walk on $\Z$, which we state in the form
found in Lawler and Limic~\cite[Proposition 2.5.3, Corollary 2.5.4]{LL}: As $k\to\infty$,
\begin{equation}
\label{LLT}
\P_0(\zeta_k=m)= \sqrt{\frac{2}{\pi k}}
\exp\big(-\frac{m^2}{2k}\big)\,\exp\Big(O\big(\frac{1}{k}+ \frac{m^4}{k^3}\big)\Big)
\end{equation}
uniformly over integers $m$ such that $|m|\leq k$ and $k+m$ is even.

We fix $\alpha\in(0,1/4)$ and to 
simplify notation, we write $p_n=\lfloor n^{\frac{1}{2}-\alpha}\rfloor$
for every integer $n\geq 1$. Recall the notation
$\tau_p=\inf\{n\geq 0: \zeta_n = \zeta_0-p\}$.

\begin{lemma}
\label{key4D1}
If $\eta>0$ is sufficiently small, we have
$$\lim_{n\to\infty}
\Bigg(\; \sup_{\begin{subarray}{c}
n^{1-\eta} \leq k \leq 2n\\
n^{\frac{1}{2}-\frac{\alpha}{2}}\leq m \leq n^{\frac{1}{2}+\eta}
\end{subarray}}\Big|(\log n)\,\P_m\Big( \wh W_j \not = \wh W_0\;, \forall j\in\llbracket 1, \tau_{p_n}\rrbracket
\,\Big|\, T=k\Big) - \frac{4\pi^2\sigma^4}{1-2\alpha}\Big|\;\Bigg)
=0,$$
where in the supremum we consider only integers $m$ and $k$ such that
$k+m$ is odd.
\end{lemma}

\proof 
We first explain how to choose $\eta$.
We set $q_n=\lfloor n^{1-\frac{3\alpha}{2}}\rfloor$ and note that
$$\P_0\big(\zeta_{q_n}>n^{\frac{1}{2}-\frac{\alpha}{2}}\big)\leq
\P_0\big(\zeta_{q_n} > q_n^{\frac{1}{2}+c(\alpha)}\big)$$
where $c(\alpha)= \frac{\alpha}{4-6\alpha}>0$. By a standard bound,
the latter probability is bounded (for $n$ large) by
$\exp(-n^{\gamma})$, where the constant $\gamma=\gamma(\alpha)>0$
only depends on $\alpha$. We fix $\eta>0$ such that $3\eta <\gamma$ and
$\eta\in(0,\alpha/8)$.

To simplify notation, we then set
$$\Delta_n \colonequals \big\{(m,k): n^{\frac{1}{2}-\frac{\alpha}{2}}\leq m \leq n^{\frac{1}{2}+\eta},\;
n^{1-\eta} \leq k \leq 2n\hbox{ and } k+m \hbox{ is odd}\big\}.$$
Since $p_n \sim n^{-\alpha/4}\,\sqrt{q_n}$, standard estimates give, for every 
$\delta\in(0,\frac{\alpha}{4})$,
\begin{equation}
\label{exitRW}
\lim_{n\to\infty} n^{\delta}\, \P_0(\tau_{p_n} \geq q_n) = 0.
\end{equation}
We claim that we have also, for every $\delta\in(0,\frac{\alpha}{4})$,
\begin{equation}
\label{keyeq00}
\lim_{n\to\infty} n^\delta\,\sup_{(m,k)\in \Delta_n}\;\P_m(\tau_{p_n} \geq q_n \mid T=k)= 0.
\end{equation}
Let us postpone the proof of~\eqref{keyeq00} and derive 
the estimate of the lemma. 

\begin{figure}[!htbp]
 \begin{center}
 \includegraphics[width=14.5cm]{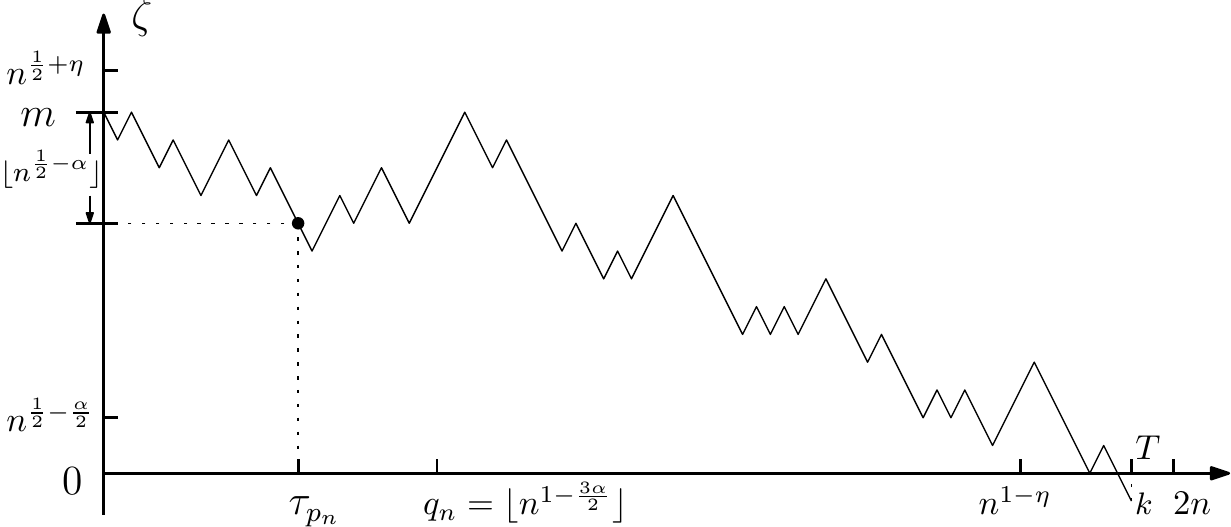}
 \caption{\textsf{Illustration of the proof of Lemma \ref{key4D1}} \label{fig-abscont}}
 \end{center}
\end{figure}

Let us consider $(m,k)\in \Delta_n$. We have
\begin{align}
\label{keyeq11}
&\P_m\Big(\{\wh W_j \not = \wh W_0\;, \forall j\in\llbracket 1,\tau_{p_n}\rrbracket\}\cap\{\tau_{p_n} \leq q_n\} \cap\{T=k\}\Big)\notag\\
&\quad=\E_m\Big[\ind{\{\tau_{p_n} \leq q_n\}}\,
\ind{\{\wh W_j \not = \wh W_0\;, \forall j\in\llbracket 1,\tau_{p_n}\rrbracket\}}
\P_{m-p_n}(T=k-\ell)_{\ell=\tau_{p_n}}\Big],
\end{align}
where we have used the strong Markov property at $\tau_{p_n}$. We now would like to say that the
quantity $\P_{m-p_n}(T=k-\ell)$, evaluated at $\ell=\tau_{p_n}$, does not differ too much from $\P_{m}(T=k)$
under our conditions on $m,k$ and $\tau_{p_n}$
(see Fig.\,4 for an illustration).  Let $k'$ be an
integer such that $k-q_n\leq k'\leq k$ and $k'+m-p_n$ is odd. By
Kemperman's formula,
\begin{equation}
\label{Kemp2}
\P_{m-p_n}(T=k')= \frac{m-p_n+1}{k'}\,\P_0(\zeta_{k'}=m-p_n+1)
\end{equation}
and by \eqref{LLT},
\begin{equation}
\label{keyeq3}
\P_{0}(\zeta_{k'}=m-p_n+1)
= \sqrt{\frac{2}{\pi k'}}
\exp\Big(-\frac{(m-p_n+1)^2}{2k'}\Big)\,\exp\!\Big(O(\frac{1}{k'}+\frac{(m-p_n+1)^4}{{k'}^3})\Big).
\end{equation}
Next observe that
\begin{align*}
\Big|\frac{(m-p_n+1)^2}{2k'}- \frac{(m+1)^2}{2k}\Big|
&\leq \frac{(m+1)^2}{2} \big(\frac{1}{k'}-\frac{1}{k}\big)
+ \frac{(m+1)^2 -(m-p_n+1)^2}{2k'}\\
&\leq \frac{q_n(m+1)^2}{kk'} + \frac{p_n(m+1)}{k'}
\end{align*}
which tends to $0$ as $n\to\infty$, uniformly in $m,k,k'$. 
Comparing the estimate for $\P_{m-p_n}(T=k')$
that follows from~\eqref{Kemp2} and~\eqref{keyeq3}
with the similar estimate for $\P_m(T=k)$ that follows from~\eqref{Kemp} and~\eqref{LLT}, we get
$$\lim_{n\to\infty} \left(
\sup_{m,k,k'} \left| \frac{\P_{m-q_n}(T=k')}{\P_m(T=k)} -1 \right|\right)= 0,$$
where the supremum is over all choices of $(m,k,k')$ 
such that $(m,k)\in \Delta_n$ and $k'$ satisfies the preceding conditions. 
Using~\eqref{keyeq11}, we obtain that, for any fixed $\delta>0$, we have for all sufficiently large $n$, for every $(m,k)\in \Delta_n$,
\begin{align*}
&(1-\delta)\,\P_m\big(\{\tau_{p_n} \leq q_n\} \cap\{\wh W_j \not = \wh W_0\;, \forall j\in\llbracket 1,\tau_{p_n}\rrbracket\}\big)\\
&\quad\leq  \P_m\big(\{\tau_{p_n} \leq q_n\} \cap\{\wh W_j \not = \wh W_0\;, \forall j\in\llbracket 1,\tau_{p_n}\rrbracket\}\mid T=k\big)\\ 
&\quad\leq (1+\delta)\,\P_m\big(\{\tau_{p_n} \leq q_n\} \cap\{\wh W_j \not = \wh W_0\;, \forall j\in\llbracket 1,\tau_{p_n}\rrbracket\}\big).
\end{align*}
The quantity $\P_m\big(\{\tau_{p_n} \leq q_n\} \cap\{\wh W_j \not = \wh W_0\;, \forall j\in\llbracket 1,\tau_{p_n}\rrbracket\}\big)$ does not depend on $m\in\Z$, and 
$(\log n)\, \P_0(\tau_{p_n} > q_n)$ tends to $0$ by~\eqref{exitRW}.
Using Proposition~\ref{firstesti}, we have thus
$$\lim_{n\to\infty} (\log n)\,
\P_0\big(\{\tau_{p_n} \leq q_n\} \cap\{\wh W_j \not = \wh W_0\;, \forall j\in\llbracket 1,\tau_{p_n}\rrbracket\}\big)=\frac{4\pi^2\sigma^4}{1-2\alpha}.$$
The estimate 
of the lemma follows from the preceding considerations and~\eqref{keyeq00}.

It remains to prove~\eqref{keyeq00}. 
If $(m,k)\in \Delta_n$, we have
$$\P_m(\tau_{p_n} \geq q_n \mid T=k)=
\frac{\P_m(\{\tau_{p_n} \geq q_n\}\cap\{T=k\})}
{\P_m(T=k)}.$$
Recall formula~\eqref{Kemp} for $\P_m(T=k)$ and also note that by~\eqref{LLT}, 
\begin{equation}
\label{keyeq1}
\P_0(\zeta_k=m+1)= \sqrt{\frac{2}{\pi k}}
\exp\Big(-\frac{(m+1)^2}{2k}\Big)\,\exp\!\Big( O(\frac{1}{k}+\frac{m^4}{k^3})\Big),
\end{equation}
when $n\to\infty$, uniformly in $(m,k)\in \Delta_n$. Notice that $\frac{1}{k}+\frac{m^4}{k^3}\longrightarrow 0$ as $n\to\infty$,
 uniformly in $(m,k)\in \Delta_n$, and that $\frac{m^2}{2k}\leq n^{3\eta}$
 if $(m,k)\in \Delta_n$. By our choice of $\eta$, it follows that
\begin{equation}
\label{keyeq0}
\P_m(\zeta_{q_n}>m+n^{\frac{1}{2}-\frac{\alpha}{2}}\mid T=k)
\leq \frac{\P_m(\zeta_{q_n}>m+n^{\frac{1}{2}-\frac{\alpha}{2}})}{\P_m(T=k)}
=\frac{k}{m+1} \frac{\P_0(\zeta_{q_n}>n^{\frac{1}{2}-\frac{\alpha}{2}})}{\P_0(\zeta_k=m+1)}
=O\big(\frac{1}{n}\big)
\end{equation}
as $n\to\infty$, uniformly in $m$
and $k$. 

On the other hand, by applying the Markov property at time $q_n$, we have
$$\P_m\big(\{\tau_{p_n} \geq q_n\}\cap\{\zeta_{q_n}\leq m+n^{\frac{1}{2}-\frac{\alpha}{2}}\}  \cap \{T=k\}\big)
=\E_m\Big[ \ind{\{\tau_{p_n} \geq q_n\}\cap \{\zeta_{q_n}\leq m+n^{\frac{1}{2}-\frac{\alpha}{2}}\}}
\,\P_{\zeta_{q_n}}( T = k-q_n)\Big].$$
On the event $\{\tau_{p_n} \geq q_n\}\cap\{\zeta_{q_n}\leq m+n^{\frac{1}{2}-\frac{\alpha}{2}}\}$ we have 
$m- p_n\leq \zeta_{q_n}\leq m+n^{\frac{1}{2}-\frac{\alpha}{2}}$, $\P_m$ a.s. If
$m-p_n\leq m'\leq m+n^{\frac{1}{2}-\frac{\alpha}{2}}$ and $m'+k-q_n$ is odd, using again Kemperman's formula, we have
$$\P_{m'}(T = k-q_n) = \frac{m'+1}{k-q_n}
\P_{m'}(\zeta_{k-q_n}=-1)= \frac{m'+1}{k-q_n}
\P_{0}(\zeta_{k-q_n}=m'+1).$$
Furthermore, from~\eqref{LLT},
\begin{equation}
\label{keyeq2}
\P_{0}(\zeta_{k-q_n}=m'+1)
= \sqrt{\frac{2}{\pi (k-q_n)}}
\exp\Big(-\frac{(m'+1)^2}{2(k-q_n)}\Big)\,\exp\!\Big(O(\frac{1}{k}+\frac{m'^4}{k^3})\Big).
\end{equation}
Now observe that 
$$-\frac{(m'+1)^2}{2(k-q_n)} + \frac{(m+1)^2}{2k}
\leq -\frac{(m'+1)^2-(m+1)^2}{2k}=- \frac{(m'-m)(m'+m+2)}{2k}$$
and the right-hand side tends to $0$ as $n\to\infty$, uniformly in
$(m,k)\in \Delta_n$ and $m'$ such that $m-p_n\leq m'\leq m+n^{\frac{1}{2}-\frac{\alpha}{2}}$.
By comparing~\eqref{keyeq1} and~\eqref{keyeq2}, noting that 
$m'\leq 2m$ under our assumptions, we get
$$\limsup_{n\to\infty}\left( \sup_{m,k,m'} \frac{\P_{m'}(T = k-q_n)}{\P_m(T=k)}
\right)\leq 2.$$
It follows that, for all $n$ sufficiently large, we have, for every 
$(m,k)\in \Delta_n$,
$$ \frac{\P_m\big(\{\tau_{p_n} \geq q_n\}\cap\{\zeta_{q_n}\leq m+n^{\frac{1}{2}-\frac{\alpha}{2}}\}  \cap \{T=k\}\big)}{\P_m(T=k)}
\leq 3\, \P_m(\tau_{p_n} \geq q_n)= 3 \,\P_0(\tau_{p_n} \geq q_n).$$
Recalling~\eqref{exitRW}, we have thus proved that, for every $\delta\in(0,\frac{\alpha}{4})$,
$$\lim_{n\to\infty} n^{\delta}\,\sup_{(m,k)\in \Delta_n}\P_m\big(\{\tau_{p_n} \geq q_n\}\cap\{\zeta_{q_n}\leq m+n^{\frac{1}{2}-\frac{\alpha}{2}}\}  \mid T=k\big)= 0,$$
and by combining this with~\eqref{keyeq0}, we get
the desired estimate~\eqref{keyeq00}.
 \endproof
 
 \medskip
 We set, for every $n\geq 1$,
 $$R^\bullet_n \colonequals \#\big\{\wh W^{(n)}_0,\wh W^{(n)}_1,\ldots, \wh W^{(n)}_{2n}\big\}.$$
 
 \begin{proposition}
 \label{secondmo}
 We have
 $$\limsup_{n\to\infty} \Big(\frac{\log n}{n}\Big)^2 \,\E\big[(R^\bullet_n)^2\big] \leq (8\pi^2\sigma^4)^2.$$
 \end{proposition}
 
 \proof We note that
 $$R^\bullet_n= \sum_{i=1}^{2n}
 \ind{\{\wh W^{(n)}_\ell \not = \wh W^{(n)}_i\;, \forall \ell\in\llbracket i+1,2n\rrbracket\}}$$
 and therefore 
 \begin{equation}
 \label{expanmo2}
 \E\big[(R^\bullet_n)^2\big]= \sum_{i,j=1}^{2n} \P(A_n(i,j)),
 \end{equation}
 where $A_n(i,j)$ is defined by
$$A_n(i,j)\colonequals \big\{\wh W^{(n)}_\ell \not = \wh W^{(n)}_i\;, \forall \ell\in\llbracket i+1,2n\rrbracket\big\}
\cap \big\{\wh W^{(n)}_\ell \not = \wh W^{(n)}_j\;, \forall \ell\in\llbracket j+1,2n\rrbracket\big\}.$$

We fix $\alpha\in(0,1/4)$, and define $p_n$ and $q_n$
for every $n\geq 1$ as above. We also fix $\eta>0$ so that the conclusion 
of Lemma~\ref{key4D1} holds. 

In view of proving the proposition, we will use formula~\eqref{expanmo2}.
In this formula, we can restrict our attention to values
of $i$ and $j$ such that $j-i > n^{1-\frac{\alpha}{2}}$ and $j <2n-n^{1-\eta}$
(or the same with $i$ and $j$ interchanged). 
Also, when bounding $\P(A_n(i,j))$, we may impose the additional constraint
that $n^{\frac{1}{2}-\frac{\alpha}{2}}\leq \zeta^{(n)}_i\leq n^{\frac{1}{2}+\eta}$
and $n^{\frac{1}{2}-\frac{\alpha}{2}}\leq \zeta^{(n)}_j\leq n^{\frac{1}{2}+\eta}$:
Indeed, Lemma~\ref{estimate-excu} readily shows that the event where either of these
constraints is not satisfied will give a negligible contribution to the sum in~\eqref{expanmo2}.

Let us fix $i,j\in\llbracket 1,2n\rrbracket$ such that $j-i > n^{1-\frac{\alpha}{2}}$ and $j <2n-n^{1-\eta}$.
By using the definition of $W^{(n)}$ as a conditioned process and applying the Markov property at time $i$, we have
\begin{align}
\label{moment2tech1}
&\P\big(A_n(i,j)\cap\{ n^{\frac{1}{2}-\frac{\alpha}{2}}\leq \zeta^{(n)}_i\leq n^{\frac{1}{2}+\eta} \}\cap \{n^{\frac{1}{2}-\frac{\alpha}{2}}\leq \zeta^{(n)}_j\leq n^{\frac{1}{2}+\eta}\}
\big)\notag\\
&= \frac{\E\Big[\ind{\{n^{\frac{1}{2}-\frac{\alpha}{2}}\leq \zeta_i\leq n^{\frac{1}{2}+\eta} \}}\ind{\{T>i\}}\, \E_{\zeta_i}\Big[\ind{\{n^{\frac{1}{2}-\frac{\alpha}{2}}\leq \zeta_{j-i}\leq n^{\frac{1}{2}+\eta} \}}\ind{A'_n(i,j)}\ind{\{T=2n+1-i\}}\Big]\Big]}{\P(T=2n+1)}
\end{align}
where 
$$A'_n(i,j)\colonequals \big\{\wh W_\ell \not = 
\wh W_0\;, \forall \ell\in\llbracket1,2n-i\rrbracket\big\}\cap 
\big\{\wh W_\ell \not = \wh W_{j-i}\;, \forall \ell\in\llbracket j-i+1,2n-i\rrbracket\big\}.$$

Setting $r=j-i$, we are thus led to bound
\begin{equation}
\label{moment2tech0}
\E_m\Big[ \ind{\{\wh W_\ell \not = 
\wh W_0\;, \forall \ell\in\llbracket1,k-1\rrbracket\}}\,\ind{\{\wh W_\ell \not = \wh W_r\;, \forall \ell\in\llbracket r+1,k-1\rrbracket\}}\,\ind{\{n^{\frac{1}{2}-\frac{\alpha}{2}}\leq \zeta_r\leq n^{\frac{1}{2}+\eta}\}}\,\ind{\{T=k\}}\Big]
\end{equation}
where $n^{\frac{1}{2}-\frac{\alpha}{2}}\leq m \leq n^{\frac{1}{2}+\eta}$, $r>n^{1-\frac{\alpha}{2}}$ and $r+ n^{1-\eta} <k\leq 2n$ (and moreover $k+m$
needs to be odd). Recall the notation $\tau_{p_n}$, and set
$$\tau_{p_n}^{(r)}\colonequals \inf\{\ell\geq r: \zeta_\ell=\zeta_{r}-p_n\}.$$
Thanks to~\eqref{keyeq00},  we can also introduce the constraint
$\tau_{p_n}\leq q_n$ inside the expectation in~\eqref{moment2tech0}, up
to an error that is bounded above by $\P_m(T=k) \,o(n^{-\delta})$ for some $\delta>0$ (here the term $o(n^{-\delta})$ is uniform in $m,r,k$ 
satisfying the preceding conditions).
Furthermore, we get an upper bound by replacing the interval
$\llbracket1,k-1\rrbracket$, resp. $\llbracket r+1,k-1\rrbracket$, 
by $\llbracket 1, \tau_{p_n}\rrbracket$, resp. $\llbracket r+1, \tau^{(r)}_{p_n}\rrbracket$.
Next, using the Markov property at time $r$, and noting that $r>q_n$, we have
\begin{align*}
&\E_m\Big[ \ind{\{\wh W_\ell \not = 
\wh W_0\;, \forall \ell\in\llbracket 1, \tau_{p_n}\rrbracket\}}\,
\ind{\{\tau_{p_n}\leq q_n\}}\,\ind{\{\wh W_\ell \not = \wh W_r\;, \forall \ell\in\llbracket r+1, \tau^{(r)}_{p_n}\rrbracket\}}\,\ind{\{n^{\frac{1}{2}-\frac{\alpha}{2}}\leq \zeta_r\leq n^{\frac{1}{2}+\eta}\}}\,\ind{\{T=k\}}\Big]\\
&=\E_m\Big[ \ind{\{\wh W_\ell \not = 
\wh W_0\;, \forall \ell\in\llbracket 1, \tau_{p_n}\rrbracket\}}\,
\ind{\{\tau_{p_n}\leq q_n\}}\,\ind{\{T>r\}}\,\ind{\{n^{\frac{1}{2}-\frac{\alpha}{2}}\leq \zeta_r\leq n^{\frac{1}{2}+\eta}\}}\,
\E_{(W_r)}\Big[\ind{\{\wh W_\ell \not = 
\wh W_0\;, \forall \ell\in\llbracket 1, \tau_{p_n}\rrbracket\}}\,\ind{\{T=k-r\}}\Big]\Big]
\end{align*}
See Fig.\,5 for an illustration. 

\begin{figure}[!htbp]
 \begin{center}
 \includegraphics[width=14.5cm]{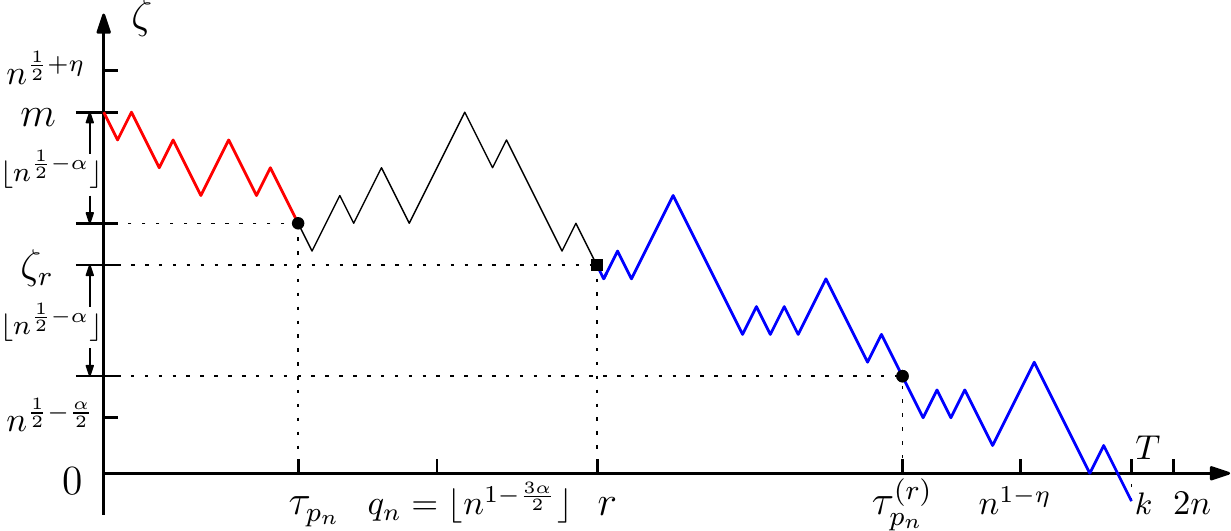}
 \caption{\textsf{Illustration of the proof of Proposition \ref{secondmo}.} \label{fig-upperbound}}
 \end{center}
\end{figure}
Then the key observation is the following. Let $z_0=m,z_1,\ldots,z_r$
be a simple random walk trajectory over $\llbracket 0,r\rrbracket$ such that
$0\leq \min\{z_\ell:0\leq \ell\leq r\} \leq m-p_n$. Then under $\P_m$,
conditionally on the event $\{\zeta_1=z_1,\ldots,\zeta_r=z_r\}$, the path
$(W_r(z_r)-W_r(z_r-\ell))_{\ell\geq 0}$ is {\it independent} of the event $\{\wh W_\ell \not = 
\wh W_0\;, \forall \ell\in\llbracket 1, \tau_{p_n}\rrbracket\}$, and distributed as 
$(S_\ell)_{\ell\geq 0}$. This property easily follows from
the construction of the discrete snake.

Thanks to the latter observation, we may rewrite the right-hand side of the
last display, after conditioning with respect to $\zeta_1,\ldots,\zeta_r$,
in the form
\begin{equation}
\label{keytech100}
\E_m\Big[ \ind{\{\wh W_\ell \not = 
\wh W_0\;, \forall \ell\in\llbracket 1, \tau_{p_n}\rrbracket\}}\,
\ind{\{\tau_{p_n}\leq q_n\}}\,\ind{\{T>r\}}\,\ind{\{n^{\frac{1}{2}-\frac{\alpha}{2}}\leq \zeta_r\leq n^{\frac{1}{2}+\eta}\}}
\E_{\zeta_r}\Big[\ind{\{\wh W_\ell \not = 
\wh W_0\;, \forall \ell\in\llbracket 1, \tau_{p_n}\rrbracket\}}\,\ind{\{T=k-r\}}\Big]\Big].
\end{equation}
Recall that $k-r>n^{1-\eta}$, and let $\ve>0$. It follows from Lemma~\ref{key4D1} that, 
for $n$ large enough, on the event
$\{n^{\frac{1}{2}-\frac{\alpha}{2}}\leq \zeta_r\leq n^{\frac{1}{2}+\eta}\}$, the quantity 
$$ \E_{\zeta_r}\Big[\ind{\{\wh W_\ell \not = 
\wh W_0\;, \forall \ell\in\llbracket 1, \tau_{p_n}\rrbracket\}}\,\ind{\{T=k-r\}}\Big]$$
is bounded above by 
$$\Big(\frac{4\pi^2\sigma^4}{1-2\alpha} +\ve\Big)(\log n)^{-1}\,\P_{\zeta_r}(T=k-r).$$ 
Hence the quantity~\eqref{keytech100} is also bounded by
\begin{align*}
&\Big(\frac{4\pi^2\sigma^4}{1-2\alpha} +\ve\Big)(\log n)^{-1}\,
\E_m\Big[ \ind{\{\wh W_\ell \not = 
\wh W_0\;, \forall \ell\in\llbracket 1, \tau_{p_n}\rrbracket\}}\,
\ind{\{\tau_{p_n}\leq q_n\}}\,\ind{\{T>r\}}\,\ind{\{n^{\frac{1}{2}-\frac{\alpha}{2}}\leq \zeta_r\leq n^{\frac{1}{2}+\eta}\}}\,\P_{\zeta_r}(T=k-r)\Big]\\
&\quad=\Big(\frac{4\pi^2\sigma^4}{1-2\alpha} +\ve\Big)(\log n)^{-1}\, \E_m\Big[ \ind{\{\wh W_\ell \not = 
\wh W_0\;, \forall \ell\in\llbracket 1, \tau_{p_n}\rrbracket\}}\,
\ind{\{\tau_{p_n}\leq q_n\}}\,\ind{\{n^{\frac{1}{2}-\frac{\alpha}{2}}\leq \zeta_r\leq n^{\frac{1}{2}+\eta}\}}\,\ind{\{T=k\}}\Big],
\end{align*}
again by the Markov property at time $r$. Finally, another application of
Lemma~\ref{key4D1} shows that the quantity in the last display is bounded above 
for $n$ large by
$$\Big(\frac{4\pi^2\sigma^4}{1-2\alpha} +\ve\Big)^2(\log n)^{-2}\, \P_m(T=k).$$
Summarizing, we see that the quantity~\eqref{moment2tech0} is bounded
above for $n$ large by 
$$\Big(\Big(\frac{4\pi^2\sigma^4}{1-2\alpha} +\ve\Big)^2(\log n)^{-2}
+ o(n^{-\delta})\Big)\, \P_m(T=k).$$
Finally, from~\eqref{moment2tech1}, we have for $n$ large
\begin{align*}
&\P\big(A_n(i,j)\cap\{ n^{\frac{1}{2}-\frac{\alpha}{2}}\leq \zeta^{(n)}_i\leq n^{\frac{1}{2}+\eta} \}\cap \{n^{\frac{1}{2}-\frac{\alpha}{2}}\leq \zeta^{(n)}_j\leq n^{\frac{1}{2}+\eta}\}
\big)\\
&\leq \Big(\Big(\frac{4\pi^2\sigma^4}{1-2\alpha} +\ve\Big)^2(\log n)^{-2}
+ o(n^{-\delta})\Big)\,
\frac{\E\Big[\ind{\{T>i\}}\, \P_{\zeta_i}(T=2n+1-i)\Big]}{\P(T=2n+1)}\\
&= \Big(\frac{4\pi^2\sigma^4}{1-2\alpha} +\ve\Big)^2(\log n)^{-2}
+ o(n^{-\delta}),
\end{align*}
where the term $o(n^{-\delta})$ is uniform in $i$ and $j$ satisfying
the preceding conditions. 
The statement of the proposition follows by summing this bound over $i$ and $j$. \endproof

\begin{lemma}
\label{estimoment1}
We have
$$\liminf_{n\to\infty} \,\frac{\log n}{n}\,\E\big[R^\bullet_n\big]\geq 8\pi^2\sigma^4.$$
\end{lemma}

\proof Let $\delta>0$ and $\ve\in(0,\frac{1}{2})$. To simplify notation we write
$n(\ve)=\lfloor 2(1-2\ve)n\rfloor$ in this proof. We fix $0<a<b$ such that, if
$(\be_t)_{0\leq t\leq 1}$ denotes a normalized Brownian excursion defined under
the probability measure $P$, we have
$$P(\be_\ve\notin (a,b)\big)=P(\be_{1-\ve}\notin (a,b)\big)< \delta.$$
Since we know that
the sequence of processes
$((2n)^{-1/2} \zeta^{(n)}_{\lfloor 2nt\rfloor})_{0\leq t\leq 1}$ converges 
in distribution to $(\be_t)_{0\leq t\leq 1}$, it follows that, for every 
sufficiently large $n$,
\begin{equation}
\label{moment1tech2}
\P\Big(\zeta^{(n)}_{\lfloor 2n\ve\rfloor}
\notin [a\sqrt{2n},b\sqrt{2n}]\hbox{ or } \zeta^{(n)}_{\lfloor 2n\ve\rfloor+n(\ve)}
\notin [a\sqrt{2n},b\sqrt{2n}]\Big)\leq \delta.
\end{equation}
Let $\mu^{(n)}_\ve$ denote the law of $\zeta^{(n)}_{\lfloor 2n\ve\rfloor}$. If
$F_n$ is a nonnegative function on $\Z^{n(\ve)+1}$, the Markov property
gives
\begin{align*}
&\E\Big[\ind{\{a\sqrt{2n}\leq \zeta^{(n)}_{\lfloor 2n\ve\rfloor}\leq b\sqrt{2n}\}}\,
\ind{\{a\sqrt{2n}\leq \zeta^{(n)}_{\lfloor 2n\ve\rfloor+n(\ve)}\leq b\sqrt{2n}\}}\,
F_n\big((\zeta^{(n)}_{\lfloor 2n\ve\rfloor+k})_{0\leq k\leq n(\ve)}\big)\Big]\\
&=\E\Big[\ind{\{a\sqrt{2n}\leq \zeta^{(n)}_{\lfloor 2n\ve\rfloor}\leq b\sqrt{2n}\}}\,
\E_{\zeta^{(n)}_{\lfloor 2n\ve\rfloor}}\Big[\ind{\{a\sqrt{2n}\leq \zeta_{n(\ve)}\leq b\sqrt{2n}\}}
F_n\big((\zeta_{k})_{0\leq k\leq n(\ve)}\big)\,\Big|\, T=2n+1-\lfloor 2n\ve\rfloor\Big]\Big]\\
&=\sum_{a\sqrt{2n}\leq m\leq b\sqrt{2n}} \mu^{(n)}_\ve(m)\,
\frac{\E_m\Big[\ind{\{a\sqrt{2n}\leq \zeta_{n(\ve)}\leq b\sqrt{2n}\}}
F_n\big((\zeta_{k})_{0\leq k\leq n(\ve)}\big)\,\ind{\{T=2n+1-\lfloor 2n\ve\rfloor\}}\Big]}
{\P_m(T=2n+1-\lfloor 2n\ve\rfloor)}\\
&=\sum_{a\sqrt{2n}\leq m\leq b\sqrt{2n}} \mu^{(n)}_\ve(m)\,
\frac{\E_m\Big[\ind{\{a\sqrt{2n}\leq \zeta_{n(\ve)}\leq b\sqrt{2n}\}}
F_n\big((\zeta_{k})_{0\leq k\leq n(\ve)}\big)\, \ind{\{T>n(\ve)\}}\,\P_{\zeta_{n(\ve)}}(T=\wt n(\ve))\Big]}
{\P_m(T=2n+1-\lfloor 2n\ve\rfloor)}
\end{align*}
where $\wt n(\ve)\colonequals 2n+1-\lfloor 2n\ve\rfloor-n(\ve)$.

Let $m,m'\in[a\sqrt{2n},b\sqrt{2n}]$ be such that $m+\lfloor 2n\ve\rfloor$ and 
$m'+\lfloor 2n\ve\rfloor + n(\ve)$ are even. By Kemperman's formula~\eqref{Kemp},
$$\frac{\P_{m'}(T=\wt n(\ve))}{\P_m(T=2n+1-\lfloor 2n\ve\rfloor)}
= \frac{2n+1-\lfloor 2n\ve\rfloor}{\wt n(\ve)}\, \frac{m'+1}{m+1}\,
\frac{\P_0(\zeta_{\wt n(\ve)} =m'+1)}{ \P_0(\zeta _{2n+1-\lfloor 2n\ve\rfloor}=m+1)}$$
and using~\eqref{LLT}, we easily obtain that there exists
a finite constant $C(\ve,a,b)$ such that, for every sufficiently large $n$,
and every $m,m'$ satisfying the above conditions,
$$\frac{\P_{m'}(T=\wt n(\ve))}{\P_m(T=2n+1-\lfloor 2n\ve\rfloor)}\leq C(\ve,a,b).$$
We thus obtain that, for every large enough $n$,
\begin{align*}
&\E\Big[\ind{\{a\sqrt{2n}\leq \zeta^{(n)}_{\lfloor 2n\ve\rfloor}\leq b\sqrt{2n}\}}\,
\ind{\{a\sqrt{2n}\leq \zeta^{(n)}_{\lfloor 2n\ve\rfloor+n(\ve)}\leq b\sqrt{2n}\}}\,
F_n\big((\zeta^{(n)}_{\lfloor 2n\ve\rfloor+k})_{0\leq k\leq n(\ve)}\big)\Big]\\
&\leq C(\ve,a,b) \sum_{a\sqrt{2n}\leq m\leq b\sqrt{2n}} \mu^{(n)}_\ve(m)\,
\E_m\Big[\ind{\{a\sqrt{2n}\leq \zeta_{n(\ve)}\leq b\sqrt{2n}\}}
F_n\big((\zeta_{k})_{0\leq k\leq n(\ve)}\big)\, \ind{\{T>n(\ve)\}}\Big]\\
&\leq  C(\ve,a,b) \sum_{a\sqrt{2n}\leq m\leq b\sqrt{2n}} \mu^{(n)}_\ve(m)\,
\E_m\Big[F_n\big((\zeta_{k})_{0\leq k\leq n(\ve)}\big)\Big].
\end{align*}
Let $G_n$ be a nonnegative measurable function on $\mathcal{W}^{n(\ve)+1}$. 
The preceding bound remains valid if we replace $F_n((\zeta^{(n)}_{\lfloor 2n\ve\rfloor+k})_{0\leq k\leq n(\ve)})$ by $G_n((W^{(n)}_{\lfloor 2n\ve\rfloor+k})_{0\leq k\leq n(\ve)})$ in the left-hand side and 
$F_n((\zeta_{k})_{0\leq k\leq n(\ve)})$ by $G_n((W_{k})_{0\leq k\leq n(\ve)})$ in the right-hand side (just use the fact that the conditional distribution 
of $W^{(n)}$  given $\zeta^{(n)}$ is the same as the conditional 
distribution of $W$ given $\zeta$). In particular, if we let
$G_n(w_0,w_1,\ldots,w_{n(\ve)})$ be the indicator function of the set where
$$\Big|\frac{\log n(\ve)}{n(\ve)}\#\big\{\wh w_0,\wh w_1,\ldots,\wh w_{n(\ve)}\big\} - 4\pi^2\sigma^4\Big| > \delta,$$
we obtain that
\begin{align}
\label{moment1tech3}
&\P\Big(\zeta^{(n)}_{\lfloor 2n\ve\rfloor}\in[a\sqrt{2n},b\sqrt{2n}],\;
\zeta^{(n)}_{\lfloor 2n\ve\rfloor+n(\ve)}\in[a\sqrt{2n},b\sqrt{2n}],\;
\Big|\frac{\log n(\ve)}{n(\ve)}\,R^{\bullet,\ve}_n - 4\pi^2\sigma^4\Big| > \delta\Big)
\notag\\
&\leq C(\ve,a,b)\,\P\Big(\Big|\frac{\log n(\ve)}{n(\ve)}\,R_{n(\ve)} - 4\pi^2\sigma^4\Big| > \delta\Big),
\end{align}
where 
$$R^{\bullet,\ve}_n \colonequals \#\big\{\wh W^{(n)}_{\lfloor 2n\ve\rfloor},
\wh W^{(n)}_{\lfloor 2n\ve\rfloor +1},\ldots, \wh W^{(n)}_{\lfloor 2n\ve\rfloor+n(\ve)}\big\}.$$
Here we used the (obvious) fact that the distribution of $R_n$ under $\P_m$
does not depend on $m$. 

By Theorem~\ref{freesnake4D}, the right-hand side of~\eqref{moment1tech3}
tends to $0$ as $n\to\infty$. Using also~\eqref{moment1tech2}, we obtain
that
$$\limsup_{n\to\infty} \P\Big(\Big|\frac{\log n(\ve)}{n(\ve)}\,R^{\bullet,\ve}_n - 4\pi^2\sigma^4\Big| > \delta\Big) \leq \delta.$$
Since $R^\bullet_n\geq R^{\bullet,\ve}_n$ and since both $\delta$ and $\ve$
can be chosen arbitrarily small, the statement of the lemma follows.
\endproof

\begin{theorem}
\label{rangecritisnake}
We have
$$\frac{\log n}{n} \,R^\bullet_n \build{\la}_{n\to\infty}^{L^2(\P)} 8\pi^2\sigma^4.$$
\end{theorem}

\proof By combining Proposition~\ref{secondmo} and Lemma~\ref{estimoment1}, we get that
\begin{align*}
&\limsup_{n\to\infty} \E\Big[(\frac{\log n}{n}R^\bullet_n- 8\pi^2\sigma^4)^2\Big]\\
&\qquad\leq\Big( \limsup_{n\to\infty} \E\big[(\frac{\log n}{n}R^\bullet_n)^2\big]\Big) - 16 \pi^2\sigma^4 
\Big(\liminf_{n\to\infty}
\E\big[\frac{\log n}{n}R^\bullet_n\big] \Big)+ (8\pi^2\sigma^4)^2 \leq 0
\end{align*}
which gives the desired result. \endproof

\medskip
Theorem \ref{rangecritisnake} and the remarks before Lemma \ref{estimate-excu} give the case $d=4$ of Theorem \ref{treeSRW}.

\subsection{Proof of the main estimate}
\label{mainest}

In this subsection, we prove Proposition~\ref{firstesti},
which was a key ingredient of the results of the previous
subsection. We first recall some basic facts. For every 
$x\in\Z^4$ and $k\geq 0$, we set
$$p_k(x)=P(S_k=x)$$
and we now denote the
Green function of the random walk $S$ by
$$G(x)=\sum_{k=0}^\infty p_k(x)$$
($G=G_\theta$ in the notation of Section 2). A standard estimate (see e.g.~\cite[Chapter 4]{LL}) states that
\begin{equation}
\label{esti-Green}
\lim_{x\to\infty} |x|^2\,G(x)= \frac{1}{2\pi^2\sigma^2}.
\end{equation}
Let $\bp$ be the period of the random walk $S$. Since $S$ is
assumed to be symmetric, we have $\bp=1$ or $2$. Then from the
local limit theorem (see e.g.~\cite[Chapter 2]{LL}), we have
\begin{equation}
\label{LLTbis}
\lim_{j\to\infty, j\in \bp\Z}\; j^2\,p_j(0)=\frac{\bp}{4\pi^2 \sigma^4}.
\end{equation}

We state our first lemma.

\begin{lemma}
\label{visitzero}
We have
$$\lim_{k\to\infty} k\,\P(\wh W_k=0)= \frac{1}{4\pi^2 \sigma^4}.$$
\end{lemma}

\proof For every integer $k\geq 0$, set
$$\underline{\zeta}_k = \min_{0\leq j\leq k} \zeta_j,$$
and
$$X_k= \zeta_k - 2\underline{\zeta}_k.$$
From the construction of the discrete snake, and the fact that $S$ is symmetric,
the conditional distribution of $\wh W_k$ knowing that $X_k=m$ is the 
law of $S_m$. Consequently,
\begin{equation}
\label{visit0tech}
\P(\wh W_k=0)= \sum_{m=0}^\infty \P(X_k=m)\,p_m(0).
\end{equation}
Asymptotics for $P(S_m=0)=p_m(0)$ are given by~\eqref{LLTbis}.
We then need to evaluate $\P(X_k=m)$. Set $\wt X_k=1+X_k$
for every $k\geq 0$. The discrete version of Pitman's theorem
(see~\cite[Lemma 3.1]{pit0}) shows that, under the probability measure $\P$,
$(\wt X_k)_{k\geq 0}$ is a Markov chain on $\{1,2,\ldots\}$ with
transition kernel $\mathrm{Q}$ given by $\mathrm{Q}(1,2)=1$ and for
every $j\geq 2$,
$$\mathrm{Q}(j,j+1)= \frac{1}{2} \frac{j+1}{j}\,,\ \mathrm{Q}(j,j-1)= \frac{1}{2} \frac{j-1}{j}.$$
This Markov chain is also the discrete $h$-transform of simple random walk
on $\Z_+$
(killed upon hitting $0$)
corresponding to $h(j)=j$. Let $(Y_k)_{k\geq 0}$ stand for a simple
random walk on $\Z$ that starts from $\ell$ under the probability measure $P_\ell$,
and let $H_0=\inf\{n\geq 0: Y_n=0\}$. It follows from the preceding
observations that, for every integer $k\geq1$ and every $m\geq 1$ 
such that $1\leq m\leq k+1$ and $k+m$ is odd,
\begin{align*}
\P(\wt X_k=m)&= m \,P_1(Y_k=m, H_0>k)\\
&=m\Big( P_0(Y_k=m-1)-P_0(Y_k=m+1)\Big)\\
&=m\times 2^{-k}\left( {k\choose \frac{k+m-1}{2}} - 
{k\choose \frac{k+m+1}{2}}\right)\\
&= \frac{2m^2}{k+m+1}\,P_0(Y_k=m-1)
\end{align*}
Hence, for every $m\geq 0$,
\begin{equation}
\label{visit0tech2}
\P(X_k=m) = \frac{2(m+1)^2}{k+m+2}\,P_0(Y_k=m).
\end{equation}
From~\eqref{visit0tech} and~\eqref{visit0tech2}, we get
$$\P(\wh W_k=0) 
= \sum_{m=0}^k \frac{2}{k+m+2}\,\Big((m+1)^2p_m(0)\Big)\, P_0(Y_k=m),$$
and the result of the lemma follows using (\ref{LLTbis}). \endproof

In the next lemma, for every integer $k\geq 0$, we use the notation $\wt W_k$
for the time-shifted path $\wt W_k=(\wt W_k(j))_{j\leq 0}$, where $\wt W_k(j)\colonequals W_k(\zeta_k +j)$, for every $j\leq 0$. 

\begin{lemma}
\label{symmetry}
Let $k\geq 1$ such that $\P(\wh W_k=0)>0$. Under the conditional probability measure $\P(\cdot\mid \wh W_k=0)$, the two
pairs $(W_0, \wt W_k)$ and $(\wt W_k, W_0)$ have the same distribution.
\end{lemma}

\begin{proof}
Write $\pi_k(i,j)$, $i,j\geq 0$ for the joint distribution under $\P$ of the pair
$$\Big(- \min_{0\leq \ell\leq k} \zeta_\ell,\zeta_k- \min_{0\leq \ell\leq k} \zeta_\ell\Big).$$
By an easy time-reversal argument, we have $\pi_k(i,j)=\pi_k(j,i)$ for every
$i,j\geq 0$. On the other hand, under $\P$, conditionally on 
$$\Big(- \min_{0\leq \ell\leq k} \zeta_\ell,\zeta_k- \min_{0\leq \ell\leq k} \zeta_\ell\Big)=(i,j)$$
we have $W_0(-i-\ell)=W_k(-i-\ell)=\wt W_k(-j-\ell)$ for every $\ell\geq 0$, and
the two random paths 
$$\big(W_0(-i+\ell)-W_0(-i)\big)_{0\leq \ell\leq i}$$
and 
$$\big(\wt W_k(-j+\ell)-\wt W_k(-j)\big)_{0\leq \ell\leq j}=\big(W_k(-i+\ell)-W_0(-i)\big)_{0\leq \ell\leq j}$$
are independent and distributed as the random walk $S$ stopped respectively at time
$i$ and at time $j$. Note that the event $\{\wh W_k=0\}$ occurs if and only if the
latter two paths have the same endpoint. The statement of the lemma easily
follows from the preceding observations and the property $\pi_k(i,j)=\pi_k(j,i)$.
\end{proof}

Let us fix $\eta\in(0,1/4)$. Thanks
to Lemma~\ref{visitzero}, we may choose $\delta>0$ small enough so that,
for every sufficiently large $n$, 
$$\sum_{k=\lfloor (1-\delta)n\rfloor}^n \P(\wh W_k=0)< \eta.$$
We then observe that
\begin{align*}
1&=\sum_{k=0}^n \P\Big(\wh W_k=0; \wh W_\ell \not=0,\forall \ell\in\llbracket k+1,n\rrbracket\Big)\\
&=\sum_{k=0}^n  \E\Big[\ind{\{\wh W_k=0\}}\, \P_{(W_k)}\Big(\wh W_\ell \not=0,\forall \ell\in\llbracket 1,n-k\rrbracket\Big)\Big]\\
&=\sum_{k=0}^n  \E\Big[\ind{\{\wh W_k=0\}}\, \P_{(W_0)}\Big(\wh W_\ell \not=0,\forall \ell\in\llbracket 1, n-k\rrbracket\Big)\Big].
\end{align*}
In the second equality, we applied the Markov property of the discrete snake at time $k$, 
and in the third one we used Lemma~\ref{symmetry}. 

From the last equalities and our choice of $\delta$, it follows that, for $n$ large,
$$\E\Bigg[\Big(\sum_{k=0}^{\lfloor (1-\delta)n\rfloor}  \ind{\{\wh W_k=0\}}\Big)\, \P_{(W_0)}\Big(\wh W_\ell \not=0,\forall \ell\in\llbracket 1,
\lfloor\delta n\rfloor\rrbracket\Big)\Bigg]
\geq 1-\eta.$$

Next fix $\ve\in(0,1/2)$ and write $n(\ve)= \lfloor n^{\frac{1}{2}+\ve}\rfloor$ to simplify notation. 
For every integer $p\geq 1$, there exists a constant $C_{p,\ve}$ such that, for every 
$n\geq 1$,
$$\P(\tau_{n(\ve)}\leq n) \leq C_{p,\ve}n^{-p}.$$
Hence, we also get, for every sufficiently large $n$,
$$\E\Bigg[\Big(\sum_{k=0}^{\tau_{n(\ve)}-1}  \ind{\{\wh W_k=0\}}\Big) \P_{(W_0)}\Big(\wh W_\ell \not=0,\forall \ell\in\llbracket 1,
\lfloor\delta n\rfloor\rrbracket\Big)\Bigg]
\geq 1-2\eta.$$
By conditioning with respect to $W_0$, we see that the left-hand side of the preceding display is equal to
$$\E\Bigg[\E_{(W_0)}\Big[\sum_{k=0}^{\tau_{n(\ve)}-1}  \ind{\{\wh W_k=0\}}\Big]\, \P_{(W_0)}\Big(\wh W_\ell \not=0,\forall \ell\in\llbracket 1,
\lfloor\delta n\rfloor\rrbracket\Big)\Bigg].$$
We now note that, for every integer $m\geq 1$,
\begin{equation}
\label{Green}
\E_{(W_0)}\Bigg[\sum_{k=0}^{\tau_{m}-1}  \ind{\{\wh W_k=0\}}\Bigg]=
2\sum_{j=0}^{m-1} G(-W_0(-j)).
\end{equation}
(we could write $G(W_0(-j))$ instead of $G(-W_0(-j))$ because $S$
is symmetric, but the preceding formula would hold also in the non-symmetric case).
To derive formula~(\ref{Green}), first consider the case $m=1$. By a standard property 
of simple random walk, we have for every integer $i\geq 0$,
$$\E_{(W_0)}\Bigg[\sum_{k=0}^{\tau_1-1} \ind{\{\zeta_k=i\}} \Bigg]= 2.$$
Then using the conditional distribution of $W$ given the lifetime process $\zeta$, we obtain
\begin{align*}
\E_{(W_0)}\Bigg[\sum_{k=0}^{\tau_{1}-1}  \ind{\{\wh W_k=0\}}\Bigg]
&=\sum_{i=0}^\infty \E_{(W_0)}\Bigg[\sum_{k=0}^{\tau_{1}-1}  \ind{\{\zeta_k=i\}}\,\ind{\{\wh W_k=0\}}\Bigg]\\
&=\sum_{i=0}^\infty \E_{(W_0)}\Bigg[\sum_{k=0}^{\tau_{1}-1}  \ind{\{\zeta_k=i\}}\Bigg]\,p_i(-W_0(0))\\
&= 2\,G(-W_0(0)).
\end{align*}
(Of course here $W_0(0)=0$, but the previous calculation holds independently of the value of
$W_0(0)$.)
The same argument shows that, for every $j\in\llbracket1,m-1\rrbracket$,
$$\E_{(W_0)}\Bigg[\sum_{k=\tau_j}^{\tau_{{j+1}}-1}  \ind{\{\wh W_k=0\}}\Bigg]
= 2\,G(-W_0(-j))$$
and formula~(\ref{Green}) follows. 

From~(\ref{Green}) and the preceding considerations, we get that, for all sufficiently large $n$,
\begin{equation}
\label{keytech1}
2\,\E\Bigg[\Big(\sum_{j=0}^{n(\ve)-1} G(-W_0(-j))\Big) \P_{(W_0)}\Big(\wh W_\ell \not=0,\forall \ell\in\llbracket 1,
\lfloor\delta n\rfloor\rrbracket\Big)\Bigg]
\geq 1-2\eta.
\end{equation}
Now recall that, under the probability measure $\P$, $(-W_0(-j))_{j\geq 0}$ has the same distribution as $(S_j)_{j\geq 0}$. At this point we
need two other lemmas.

\begin{lemma}
\label{estimoGreen}
For every integer $p\geq 1$, there exists a
constant $C(p)$ such that, for every $n\geq 2$,
$$E\Big[ \Big(\sum_{j=0}^n G(S_j)\Big)^p\Big] \leq C(p)\,(\log n)^p.$$
\end{lemma}

\proof We first observe that
$$E\Big[ \Big(\sum_{j=0}^n G(S_j)\Big)^p\Big] =
E\Big[ \Big(\sum_{j=0}^n G(S_j)\,\ind{\{|S_j|\leq n\}}\Big)^p\Big]+ o(1)$$
as $n\to\infty$, because the event where $\sup\{|S_j|:0\leq j\leq n\}>n$ has a probability 
which decreases to $0$ faster than any negative power of $n$. For every integer $k\geq 1$ and $x\in\Z^4$, set
$$G_k(x)=\sum_{i=0}^k p_i(x).$$
Using~\eqref{esti-Green} and the standard local limit theorem (see e.g.~\cite[Chapter 2]{LL}) one easily verifies  that, for every sufficiently large $n$, 
for all $x\in\Z^4$ such that $|x|\leq n$, the bound $G_{n^3}(x)\geq \frac{1}{2}G(x)$ holds.
Thanks to this observation, it is enough to bound
$$E\Big[ \Big(\sum_{j=0}^n G_{n^3}(S_j)\,\ind{\{|S_i|\leq n\}}\Big)^p\Big].$$
However, if $S'$ stands for another random walk with the same distribution
as $S$ but independent of $S$, we have
$$\sum_{j=0}^n G_{n^3}(S_j) = E\Big[\sum_{j=0}^n\sum_{i=0}^{n^3} \ind{\{S_j=S'_i\}}\,\Big|\, S\Big],$$
and by Lemma 1 in Marcus and Rosen~\cite{MR}, we know that there exists a constant 
$C'(p)$ such that, for every $n\geq 2$,
$$E\Big[\Big(\sum_{j=0}^n\sum_{i=0}^{n^3} \ind{\{S_j=S'_i\}}\Big)^p\Big]
\leq C'(p)\,(\log n)^p.$$
The desired bound follows since the conditional expectation is a contraction
in $L^p$. 
\endproof

\begin{lemma}
\label{keyestimate}
For every $\alpha>0$, there exists a constant $C_\alpha$ such that, for every integer $m\geq 2$, we have
$$P\Big(\Big|\sum_{k=0}^m G(S_k) - \frac{1}{4\pi^2 \sigma^4} \log m\Big|\geq \alpha \log m\Big)
\leq C_\alpha (\log m)^{-3/2}.$$
\end{lemma}

We postpone the proof of Lemma~\ref{keyestimate}
and complete the proof of Proposition~\ref{firstesti}. An application of H\"older's inequality gives for $p\geq 2$,
\begin{align*}
&E\Bigg[ \Big(\sum_{j=0}^m G(S_j)\Big)\,\ind{\{\sum_{j=0}^m G(S_j) \geq (\frac{1}{4\pi^2 \sigma^4}+\alpha) \log m\}}\Bigg]\\
&\quad\leq E\Big[ \Big(\sum_{j=0}^m G(S_j)\Big)^p\Big]^{1/p} \,
P\Big(\sum_{j=0}^m G(S_j) \geq (\frac{1}{4\pi^2 \sigma^4}+\alpha) \log m
\Big)^{1/q}\\
&\quad\leq C(p)^{1/p} \log m\times 
P\Big(\Big|\sum_{j=0}^m G(S_j) - \frac{1}{4\pi^2 \sigma^4} \log m\Big|\geq \alpha \log m\Big)^{1/q}
\end{align*}
where $\frac{1}{p} +\frac{1}{q}=1$
and we used Lemma~\ref{estimoGreen}. Choosing $p\geq 4$ and using Lemma~\ref{keyestimate},
we obtain that
$$\lim_{m\to\infty} E\Bigg[ \Big(\sum_{j=0}^m G(S_j)\Big)\,\ind{\{\sum_{j=0}^m G(S_j) \geq (\frac{1}{4\pi^2 \sigma^4}+\alpha) \log m\}}\Bigg]=0.$$
From~(\ref{keytech1}) and the fact that $(-W_0(-j))_{j\geq 0}$ has the same distribution as $(S_j)_{j\geq 0}$, we 
then get, for every sufficiently large $n$,
\begin{align*}
&2\,(\frac{1}{4\pi^2 \sigma^4}+\alpha) (\log n(\ve))\;\P\Big(\wh W_\ell \not=0,\forall \ell\in\llbracket1,
\lfloor\delta n\rfloor\rrbracket\Big)\\
&\geq 1-2\eta - 2\,\E\Bigg[ \Big(\sum_{j=0}^{n(\ve)-1} G(-W_0(-j))\Big)\,\ind{\{\sum_{j=0}^{n(\ve)-1} G(-W_0(-j))) \geq (\frac{1}{4\pi^2 \sigma^4}+\alpha) \log n(\ve)\}}\Bigg]\\
&\geq 1-3\eta.
\end{align*}
Since $\log n(\ve)\leq(\frac{1}{2}+\ve)\log n$, the preceding bound implies that
$$\liminf_{n\to\infty} (\log n)\,\P\Big(\wh W_\ell \not=0,\forall \ell\in\llbracket1,
\lfloor\delta n\rfloor\rrbracket\Big) \geq \frac{1-3\eta}{1+2\ve}\;(\frac{1}{4\pi^2 \sigma^4}+\alpha)^{-1}.$$
Now note that the ratio $\log \lfloor\delta n\rfloor/ \log n$ tends to $1$ as $n\to\infty$, and that
$\eta$, $\ve$ and $\alpha$ can be chosen arbitrarily small. We conclude that
$$\liminf_{n\to\infty} (\log n)\,\P\Big(\wh W_\ell \not=0,\forall \ell\in\llbracket 1,
n\rrbracket\Big)\geq 4\pi^2 \sigma^4.$$

The proof of the analogous result for the limsup behavior is similar. In the same way as we proceeded
above, we arrive at the bound
$$\E\Bigg[\Bigg(\sum_{k=0}^{n}  \ind{\{\wh W_k=0\}}\Bigg) \P_{(W_0)}\Big(\wh W_\ell \not=0,\forall \ell\in\llbracket1,
n\rrbracket\Big)\Bigg]
\leq 1.$$
At this point, we would like to replace the sum from $k=0$ to $n$ by a sum from 
$k=0$ to $\tau_{n'(\ve)-1}$, where $n'(\ve)= \lfloor n^{\frac{1}{2}-\ve}\rfloor$ for some fixed 
$\ve\in(0,1/2)$. Simple arguments give the existence of a 
constant $C'_\ve$ such that, for every integer $n\geq 1$,
$$\P( \tau_{n'(\ve)}\geq n) \leq C'_\ve\, n^{-\ve/2}.$$
We can then write
$$1\geq \E\Bigg[\Big(\sum_{k=0}^{\tau_{n'(\ve)}-1}  \ind{\{\wh W_k=0\}}\Big) \P_{(W_0)}\big(\wh W_\ell \not=0,\forall \ell\in\llbracket 1,
n\rrbracket\big)\Bigg]
- \E\Bigg[\ind{\{\tau_{n'(\ve)}\geq n\}}\,\Big(\sum_{k=0}^{\tau_{n'(\ve)}-1}  \ind{\{\wh W_k=0\}}\Big)\Bigg],$$
and by the Cauchy-Schwarz inequality, we have
\begin{equation}
\label{boundCS}
\E\Bigg[\ind{\{\tau_{n'(\ve)}\geq n\}}\,\Big(\sum_{k=0}^{\tau_{n'(\ve)}-1}  \ind{\{\wh W_k=0\}}\Big)\Bigg]
\leq (C'_\ve\, n^{-\ve/2})^{1/2} \times 
\E\Bigg[\Big(\sum_{k=0}^{\tau_{n'(\ve)}-1}  \ind{\{\wh W_k=0\}}\Big)^2\Bigg]^{1/2}.
\end{equation}
To bound the expectation in the right-hand side, one can verify that, for every 
integer $m\geq 1$,
$$\E\Bigg[\Big(\sum_{k=0}^{\tau_{m}-1}  \ind{\{\wh W_k=0\}}\Big)^2\,\Bigg|\,W_0\Bigg]
\leq 4\Bigg(\sum_{j=0}^{m-1} G(-W_0(-j))\Bigg)^2 + 4 \,\sum_{j=0}^{m-1} \Phi(-W_0(-j))$$
where, for every $x\in \Z^4$, 
$$\Phi(x)\colonequals \sum_{y\in \Z^4} G(y)G(x-y)^2.$$
The proof of the latter bound is similar to that of~\eqref{Green} above, and we leave the details to
the reader.
One then checks from~(\ref{esti-Green}) that there exists a constant $\wt C$ such that
$$\Phi(x)\leq \wt C\,(|x|\vee 1)^{-2}\,(1+\log(|x|\vee 1)), \quad\hbox{for every } x\in \Z^4.$$ It easily follows that
$$\E\Bigg[\Bigg(\sum_{k=0}^{\tau_{m}-1}  \ind{\{\wh W_k=0\}}\Bigg)^2\Bigg]=O((\log m)^2)$$
as $m\to\infty$. Consequently the right-hand side of~(\ref{boundCS}) tends to $0$ as $n\to\infty$
and if $\eta>0$ is fixed, we have, for all $n$ sufficiently large,
$$\E\Bigg[\Bigg(\sum_{k=0}^{\tau_{n'(\ve)}-1}  \ind{\{\wh W_k=0\}}\Bigg) \P_{(W_0)}\Big(\wh W_\ell \not=0,\forall \ell\in\llbracket1,
n\rrbracket\Big)\Bigg]\leq 1+\eta.$$
Just as we obtained~(\ref{keytech1}), we deduce from the latter bound that
\begin{equation}
\label{keytech2}
2\,\E\Bigg[\Bigg(\sum_{j=0}^{n'(\ve)-1} G(-W_0(-j))\Bigg) \P_{(W_0)}\Big(\wh W_\ell \not=0,\forall \ell\in\llbracket1,
n\rrbracket\Big)\Bigg]
\leq 1+\eta.
\end{equation}
Then fix $\alpha\in(0,(4\pi^2\sigma^4)^{-1})$. It follows from (\ref{keytech2}) that
\begin{align*}
&2(\frac{1}{4\pi^2 \sigma^4}-\alpha) (\log n'(\ve))\,\E\Bigg[\ind{\{\sum_{j=0}^{n'(\ve)-1} G(-W_0(-j)) \geq (\frac{1}{4\pi^2 \sigma^4}-\alpha) \log n'(\ve)\}}\,\P_{(W_0)}
\Big(\wh W_\ell \not=0,\forall \ell\in\llbracket1,
n\rrbracket\Big)\Bigg]\\
&\leq 2\,\E\Bigg[ \Big(\sum_{j=0}^{n'(\ve)-1} G(-W_0(-j))\Big)\,\P_{(W_0)}\Big(\wh W_\ell \not=0,\forall \ell\in\llbracket1,
n\rrbracket\Big)\Bigg]\\
&\leq 1+\eta.
\end{align*}
On the other hand, 
$$(\log n'(\ve))\,\P\bigg(\sum_{j=0}^{n'(\ve)-1} G(-W_0(-j)) < (\frac{1}{4\pi^2 \sigma^4}-\alpha) \log n'(\ve)\bigg)
\build{\longrightarrow}_{n\to\infty}^{} 0$$
by Lemma~\ref{keyestimate}. By combining the last two displays, we get
$$\limsup_{n\to\infty} \,2(\frac{1}{4\pi^2 \sigma^4}-\alpha) (\log n'(\ve)) \,
\P\Big(\wh W_\ell \not=0,\forall \ell\in\llbracket 1,
n\rrbracket\Big)\leq 1+\eta.$$
Since $\eta$, $\ve$ and $\alpha$ can be chosen arbitrarily small, we get
$$\limsup_{n\to\infty} \,(\log n)\,\P\Big(\wh W_\ell \not=0,\forall \ell\in\llbracket 1,
n\rrbracket\Big)\leq 4\pi^2 \sigma^4,$$
which completes the proof of the first assertion of Proposition~\ref{firstesti}. The second
assertion is an easy consequence of the first one, noting that, for every $\ve>0$, 
both $\P(\tau_p\geq p^{2+\ve})$ and $\P(\tau_p\leq p^{2-\ve})$ are $o((\log p)^{-1})$
as $p\to\infty$. 
\hfill $\square$

\medskip
\noindent{\it Proof of Lemma~\ref{keyestimate}.} The general strategy of the proof is
to derive an analogous result for Brownian motion in $\R^4$, and then to
use a strong invariance principle to transfer this result to the random walk $S$. 

We let $B=(B_t)_{t\geq 0}$ be a four-dimensional
Brownian motion started from $0$ and set $\rho_t=|B_t|$ for every $t\geq 0$, so
that $(\rho_t)_{t\geq 0}$ is a four-dimensional Bessel process
started from $0$. 
Here is the Brownian motion version of Lemma~\ref{keyestimate}.

\begin{lemma}
\label{key-estimate2}
Let $\ve>0$. There exist two constants $C(\ve)$ and $\beta(\ve)>0$ such that,
for every $t>r\geq 1$,
$$P\bigg(\bigg|\int_r^t \frac{\mathrm{d}s}{\rho_s^2} - \frac{1}{2} \log (\frac{t}{r})\bigg| > \ve \log (\frac{t}{r})\bigg)\leq C(\ve)\,(\frac{t}{r})^{-\beta(\ve)}.$$
\end{lemma}

Let us postpone the proof of Lemma~\ref{key-estimate2}. We fix 
$\alpha>0$ and consider an integer $n\geq 1$.
By an extension due to Zaitsev~\cite{Zaitsev} of the celebrated Koml\' os-Major-Tusn\' ady strong
invariance principle, we can construct on the same probability space 
the finite sequence $(S_1,\ldots,S_n)$ and the Brownian motion $(B_t)_{t\geq 0}$,
in such a way that, for some constants $c>0, c'>0$ and $K>0$ that do not depend on
$n$, we have
$$E\Big[ \exp\Big(c \max_{1\leq k\leq n} | S_k- \sigma B_k|\Big)\Big] 
\leq K\, \exp(c'\log n).$$
It readily follows that we can find constants $C>0$ and $a>0$ (again independent of $n$)
such that
$$P\Big(\max_{1\leq k\leq n} | S_k- \sigma B_k| > C \log n\Big)\leq K n^{-a}.$$
Let  $A>2$ be a constant. Then 
$$P\Big(\inf_{t\geq (\log n)^4} \sigma |B_t| \leq AC \log n\Big)
= P\Big(\inf_{t\geq1} \sigma |B_t| \leq \frac{AC}{\log n}\Big)= O((\log n)^{-2})$$
by an easy estimate. On the event
$$E_n:= \Big\{\max_{1\leq k\leq n} | S_k- \sigma B_k| \leq C \log n\Big\}
\cap \Big\{\inf_{t\geq (\log n)^4} \sigma |B_t| > AC \log n\Big\}$$
we have, for every integer $k$ such that $(\log n)^4\leq k \leq n$,
$$|S_k|\geq \sigma |B_k|- C\log n \geq (1-\eta) \sigma|B_k|$$
and 
$$|S_k|\leq \sigma |B_k|+ C\log n \leq (1+\eta) \sigma|B_k|$$
where $\eta=1/A$. We now fix $A$ so that
$\eta\in(0,\frac{1}{5})$ and $5\eta < \pi^2\sigma^4\alpha/2$. 

Recalling our estimate~(\ref{esti-Green}), we also see that (provided 
$n$ is large enough) we have on the event $E_n$,
for every integer $k$ such that $(\log n)^4\leq k \leq n$,
$$(1-3\eta)\frac{1}{2\pi^2\sigma^4}\,|B_k|^{-2} \leq G(S_k) \leq (1+3\eta) \frac{1}{2\pi^2\sigma^4}\,|B_k|^{-2}.$$
Consequently, we have on the event $E_n$,
$$(1-3\eta)\frac{1}{2\pi^2\sigma^4}\,\sum_{k=\lceil(\log n)^4\rceil}^n|B_k|^{-2} \leq 
\sum_{k=\lceil(\log n)^4\rceil}^n G(S_k) \leq (1+3\eta) \frac{1}{2\pi^2\sigma^4}\,\sum_{k=\lceil(\log n)^4\rceil}^n|B_k|^{-2}.$$
The next step is to observe that 
$$\sum_{k=\lceil(\log n)^4\rceil}^n|B_k|^{-2}$$
is close to
$$\int_{\lceil(\log n)^4\rceil}^{n+1} \frac{\mathrm{d}s}{|B_s|^2}$$
up to a set of small probability. Indeed simple estimates show that, for any choice of $\kappa>0$, we have
$$\sup_{0\leq k\leq n}\;\sup_{k\leq s\leq k+1} |B_s-B_k| \leq \kappa\log n$$
outside of a set of probability $O(n^{-1})$. 
By choosing $\kappa$ suitably, we then see that on the event 
$$\wt E_n \colonequals E_n\cap \Big\{\sup_{0\leq k\leq n}\;\sup_{k\leq s\leq k+1} |B_s-B_k| \leq \kappa\log n\Big\}$$ 
we have
$$(1-\eta)\int_{\lceil(\log n)^4\rceil}^{n+1} \frac{\mathrm{d}s}{|B_s|^2}
\leq \sum_{k=\lceil(\log n)^4\rceil}^n|B_k|^{-2}
\leq (1+\eta) \int_{\lceil(\log n)^4\rceil}^{n+1} \frac{\mathrm{d}s}{|B_s|^2},$$
and consequently
\begin{equation}
\label{discre-conti}
(1-5\eta)\frac{1}{2\pi^2\sigma^4}\int_{\lceil(\log n)^4\rceil}^{n+1} \frac{\mathrm{d}s}{|B_s|^2}\leq 
\sum_{k=\lceil(\log n)^4\rceil}^n G(S_k) \leq (1+5\eta) \frac{1}{2\pi^2\sigma^4}\,
\int_{\lceil(\log n)^4\rceil}^{n+1} \frac{\mathrm{d}s}{|B_s|^2}.
\end{equation}

We also need to bound the quantity
$$\sum_{k=0}^{\lceil(\log n)^4\rceil-1} G(S_k).$$
However, from Lemma \ref{estimoGreen} with $p=2$, we immediately get that, for 
every integer $m\geq 2$ and every $h>0$,
\begin{equation}
\label{easy-bound}
P\Big(\sum_{k=0}^m G(S_k)\geq h\Big) \leq \frac{C(2)(\log m)^2}{h^2}.
\end{equation}

Finally,
\begin{align*}
&P\Big(\Big|\sum_{k=0}^n G(S_k) - \frac{1}{4\pi^2 \sigma^4} \log n\Big|\geq \alpha \log n\Big)\\
&\quad\leq P\Big(\sum_{k=0}^{\lceil(\log n)^4\rceil} G(S_k)\geq \frac{\alpha}{2}\log n\Big)
+ P\Big(\Big|\sum_{k=\lceil(\log n)^4\rceil}^n G(S_k) - \frac{1}{4\pi^2 \sigma^4} \log n\Big|\geq \frac{\alpha}{2} \log n\Big).
\end{align*}
The first term in the right-hand side is $O((\log n)^{-3/2})$ by~(\ref{easy-bound}). On the other hand, by~(\ref{discre-conti}), the second term
is bounded by
$$P(\wt E_n^c) + P\Big(\Big| \int_{\lceil(\log n)^4\rceil}^{n+1} \frac{\mathrm{d}s}{|B_s|^2} - \frac{1}{2} \log n\Big| \geq \alpha'\log n\Big)$$
where $\alpha'=( \frac{1}{2}  \pi^2 \sigma^4 \alpha )\wedge \frac{1}{4}$
is a constant independent of $n$, which satisfies
$$(1+5\eta)(\frac{1}{2}+\alpha')\frac{1}{2\pi^2\sigma^4} <
\frac{1}{4\pi^2\sigma^4}+\frac{\alpha}{2}\hbox{ \ and \ }
(1-5\eta)(\frac{1}{2}-\alpha')\frac{1}{2\pi^2\sigma^4} > \frac{1}{4\pi^2\sigma^4}-\frac{\alpha}{2}\;.$$
(Here we use our choice of $\eta$ such that  $5\eta < \pi^2\sigma^4\alpha/2$.)
From preceding estimates, we have $P(\wt E_n^c)=O((\log n)^{-2})$.
On the other hand, Lemma~\ref{key-estimate2} implies that
$$P\Big(\Big| \int_{\lceil(\log n)^4\rceil}^{n+1} \frac{\mathrm{d}s}{|B_s|^2} - \frac{1}{2} \log n\Big| \geq \alpha'\log n\Big)
= O(n^{-b})$$
for some $b>0$. This completes the proof of Lemma~\ref{keyestimate}. 
\hfill $\square$

\medskip
\noindent{\it Proof of Lemma~\ref{key-estimate2}.} By a scaling argument, it is enough to
consider the case $r=1$, and we consider only that case. For every integer $k\geq 0$,
set
$$\gamma_k\colonequals \inf\big\{t\geq 0: \rho_t =e^k\big\}$$
and 
$$X_k:=\int_{\gamma_k}^{\gamma_{k+1}} \frac{\mathrm{d}s}{\rho_s^2}.$$
A scaling argument shows that the variables $X_k$, $k\geq 0$ are identically
distributed. Moreover, the strong Markov property of the Bessel process implies that
the variables $X_k$, $k\geq 0$ are independent. Furthermore, the absolute continuity 
relations between Bessel processes can be used to verify that these variables have small
exponential moments. More precisely, using the explicit form of the density
of the law over the time interval $[0,t]$ of the four-dimensional Bessel process started at $1$ 
with respect to Wiener measure (see question 3 in Exercise~\uppercase\expandafter{\romannumeral11}.1.22 of Revuz and Yor~\cite{RY}),
it is an easy exercise of martingale theory to verify that
$$E\big[e^{3X_0/8}\big]= E\Big[\exp \frac{3}{8}\int_{\gamma_0}^{\gamma_1} \frac{\mathrm{d}s}{\rho_s^2}\Big]
= \sqrt{e} < \infty.$$
Set
$$c_0=E[X_0]=E[X_k]$$
for every $k\geq 0$. We can apply Cram\'er's large deviation theorem to the sequence $(X_k)_{k\geq 0}$. It
follows that, for every $\delta>0$, there exists a constant $b(\delta)>0$ such that for every 
sufficiently large $n$,
\begin{equation}
\label{Cramer1}
P\Big(\Big| \int_{\gamma_0}^{\gamma_n} \frac{\mathrm{d}s}{\rho_s^2} - c_0 n\Big| > \delta n\Big) \leq \exp(-b(\delta)n).
\end{equation}
On the other hand, it is easy to verify that the variable
$$ \int_1^{\gamma_0}\frac{\mathrm{d}s}{\rho_s^2} $$
has exponential moments. Just use the above-mentioned argument
involving the density of the law of the Bessel process to verify that
$$E\Big[\exp\Big(\frac{3}{8} \int_1^{\gamma_0}\frac{\mathrm{d}s}{\rho_s^2} \Big)\Big] <\infty$$
(deal separately with the cases $1<\gamma_0$ and $\gamma_0<1$). It then follows that, for every
$\delta >0$, and for all sufficiently large $n$,
$$P\Big( \int_1^{\gamma_0}\frac{\mathrm{d}s}{\rho_s^2}  >\delta n\Big) \leq \exp(-b'(\delta)n)$$
with some constant $b'(\delta)>0$. The same bound holds for the variable
$$ \int_{e^{2m}}^{\gamma_m}\frac{\mathrm{d}s}{\rho_s^2},$$
for any integer $m\geq 0$, since this variable has the same law
as 
$$ \int_1^{\gamma_0}\frac{\mathrm{d}s}{\rho_s^2} $$
by scaling. 

By combining the latter facts with~(\ref{Cramer1}), we obtain that, for every $\delta >0$,
there exists a constant $\wt b(\delta)>0$ such that, for every sufficiently large $n$,
\begin{equation}
\label{Cramer2}
P\Big(\Big| \int_{1}^{e^{2n}} \frac{\mathrm{d}s}{\rho_s^2} - c_0 n\Big| > \delta n\Big) \leq \exp(-\wt b(\delta)n).
\end{equation}
At this stage, we can identify the constant $c_0$, since the preceding arguments also
show that
$$c_0= \lim_{n\to\infty} \frac{1}{n} \,E\Big[ \int_{1}^{e^{2n}} \frac{\mathrm{d}s}{\rho_s^2}\Big] = 1$$
by a direct calculation of $E[(\rho_s)^{-2}]=(2s)^{-1}$. Once we know that $c_0=1$, the statement
of Lemma~\ref{key-estimate2} follows from~(\ref{Cramer2}) by elementary considerations: For
every $t\geq 1$, choose $n$ such that $e^{2n}\leq t<e^{2(n+1)}$ and observe that
$$\Big\{ \int_1^t \frac{\mathrm{d}s}{\rho_s^2} - \frac{1}{2}\log t > \ve \log t\Big\}
\subseteq \Big\{ \int_1^{e^{2(n+1)}} \frac{\mathrm{d}s}{\rho_s^2} - n > 2 \ve n\Big\}$$
whereas 
$$\Big\{ \int_1^t \frac{\mathrm{d}s}{\rho_s^2} - \frac{1}{2}\log t < - \ve \log t\Big\}
\subseteq \Big\{ \int_1^{e^{2n}} \frac{\mathrm{d}s}{\rho_s^2} - n - 1 <- 2 \ve n\Big\}.$$
This completes the proof. \hfill $\square$

\section{The range of branching random walk}

In this last section, we apply the preceding results to asymptotics for
the range of branching random walk in $\Z^d$, $d\geq 4$. We assume
that the offspring distribution $\mu$ is critical and has finite variance
$\sigma_\mu^2>0$, and that the jump distribution $\theta$
is centered and has finite moments of order $d-1$ (and as usual that 
$\theta$ is not supported on a strict subgroup of $\Z^d$).

Let $M_{\rm p}(\Z^d)$ stand for the set of all finite point measures
on $\Z^d$. Let $\z=(\z_n)_{n\geq 0}$ denote the
(discrete time) branching random walk
with jump distribution $\theta$ and offspring distribution $\mu$. This is
the Markov chain with values in $M_{\rm p}(\Z^d)$, whose transition
kernel $\mathsf{Q}$ can be described as follows. If 
$$\omega= \sum_{i=1}^p \delta_{x_i} \in M_{\rm p}(\Z^d),$$
$\mathsf{Q}(\omega,\cdot)$ is the distribution of
$$\sum_{i=1}^p \sum_{j=1}^{\xi_i} \delta_{x_i+Y_{i,j}},$$
where $\xi_1,\ldots,\xi_p$ are independent
and distributed according to $\mu$ and, conditionally on
$(\xi_1,\ldots,\xi_p)$, the random variables $Y_{i,j}$, $1\leq i\leq p$,
$1\leq j\leq \xi_i$, are independent and distributed according to $\theta$. 
More informally, each particle alive at time
$n$ is replaced at time $n+1$ by a number of offspring distributed 
according to $\mu$, and the spatial position of each of these
offspring is obtained by adding a jump distributed according to $\theta$
to the position of its parent.

The range of $\z$ is then defined by
$$\mathsf{R}(\z):= \#\{x\in\Z^d: \exists n\geq 0, \z_n(x)\geq 1\}.$$
We also write $\mathsf{N}(\z)$ for the total progeny of $\z$,
$$\mathsf{N}(\z):= \sum_{n=0}^\infty \langle \z_n,1 \rangle $$
where $\langle \z_n,1 \rangle$ is the total mass of $\z_n$. It is well
known (and easy to prove using the Lukasiewisz path introduced
in the proof of Theorem \ref{mainsuper}) that $\mathsf{N}(\z)$ has the distribution 
of the hitting time of $-\langle \z_0,1\rangle$ by a random walk 
on $\Z$ with
jump distribution $\nu(k)=\mu(k+1)$, for $k=-1,0,1,\ldots$, started from $0$.

\begin{proposition}
\label{BRW5}
Suppose that $d\geq 5$. For every integer $p\geq 1$, let $\z^{(p)}$ be
a branching random walk
with jump distribution $\theta$ and offspring distribution $\mu$, such that
$\langle \z_0^{(p)},1\rangle = p$. Then,
$$\lim_{p\to\infty}\; \frac{\mathsf{R}(\z^{(p)})}{\mathsf{N}(\z^{(p)})} = c_{\mu,\theta}
\quad\hbox{in probability,}$$
where $c_{\mu,\theta}>0$ is the constant in Theorem \ref{subaddi}. Consequently,
$$ \frac{1}{p^2}\; \mathsf{R}(\z^{(p)}) \build{\la}_{p\to\infty}^{\rm (d)} \frac{c_{\mu,\theta}}{\sigma_\mu^2}\;J$$
where the positive random variable $J$ has density
$(2\pi s^3)^{-1/2}\;\exp(-\frac{1}{2s})$ on $(0,\infty)$.
\end{proposition}

\proof We may and will assume that there exists a sequence
$\t^{1},\t^{2},\ldots $ of independent random trees
distributed according to $\Pi_\mu$, such that, 
for every $p\geq 1$, the genealogy
of $\z^{(p)}$ is coded by $\t^{1},\t^{2},\ldots ,\t^{p}$, meaning that
$\t^{i}$ is the genealogical tree of the descendants of the 
$i$-th initial particle of $\z^{(p)}$, for every $p\geq 1$ and $i\in\{1,\ldots,p\}$.
Notice that we have then
$$\mathsf{N}(\z^{(p)})=\#\t^1+\cdots+\#\t^p.$$
 For every $i\in\{1,\ldots,p\}$, we will write 
$\mathcal{S}^{(p)}_i$ for the set of all spatial locations occupied by
the particles of $\z^{(p)}$ that are descendants of the $i$-th initial particle.
Note that the location 
of the $i$-th initial particle may depend on $p$. Clearly, we have
\begin{equation}
\label{bdrangeBRW}
\mathsf{R}(\z^{(p)}) \leq \#\mathcal{S}^{(p)}_1 +\cdots + \#\mathcal{S}^{(p)}_p.
\end{equation}

Let $(H_k)_{k\geq 0}$ be the height process associated with the sequence
$\t^{1},\t^{2},\ldots $ (see the proof of
Proposition \ref{rangecondi}). Then, as an easy 
consequence of \eqref{heightforest}, we have the joint convergence
in distribution
\begin{equation}
\label{heightforest2}
\Big((\frac{1}{p} H_{\lfloor p^2t\rfloor\wedge \mathsf{N}(\z^p)})_{t\geq 0}, \frac{1}{p^2}\,\mathsf{N}(\z^{(p)})\Big)
\build{\la}_{p\to\infty}^{\rm(d)} \Big((\frac{2}{\sigma_\mu} |\beta_{t\wedge
J_{1/\sigma_\mu}}|)_{t\geq 0}, J_{1/\sigma_\mu}\Big),
\end{equation}
where $\beta$ is a standard linear Brownian motion, and for every $s\geq 0$,
$J_s=\inf\{t\geq 0: L^0_t(\beta)>s\}$, where $(L^0_t(\beta))_{t\geq 0}$
is the local time process of $\beta$ at level $0$. See \cite[Section 1.4]{LG1}
for details of the derivation of \eqref{heightforest2}. 

Fix $\ve>0$. For $\alpha\in(0,1)$, let
$i_{p,1},i_{p,2},\ldots,i_{p,m_p}$ be all indices $i\in\{1,\ldots,p\}$
such that $\#\t^{i}\geq \alpha p^2$. It follows from 
\eqref{heightforest2} that, if $\alpha$
has been chosen sufficiently small, the bound 
\begin{equation}
\label{boundsizetree}
\mathsf{N}(\z^{(p)})-(\#\t^{i_{p,1}}+\cdots + \#\t^{i_{p,m_p}})=
\sum_{i\in\{1,\ldots,p\}\backslash\{i_{p,1},\ldots,i_{p,m_p}\}}
\#\t^i < \ve\,p^2
\end{equation}
will hold with probability arbitrarily close to $1$, uniformly for
all sufficiently large $p$. On the other hand, it also follows from 
\eqref{heightforest2} that $m_p$ converges
in distribution as $p\to\infty$ to a Poisson distribution with
parameter $\sigma_\mu^{-1}\,\sqrt{2/\pi \alpha}$ (here the quantity
$\sqrt{2/\pi \alpha}$ is the mass that the It\^o excursion measure
assigns to excursions of length greater than $\alpha$). In
particular, by choosing $\alpha$ even smaller if necessary, we have $P(m_p\geq 1)>1-\ve$ 
for all $p$ large enough. We now fix $\alpha>0$ so that the preceding properties hold for all $p$ large enough.

Next we observe that, conditionally on $m_p$, the trees $\t^{i_{p,1}},\ldots,
\t^{i_{p,m_p}}$ are independent and distributed according 
to $\Pi_\mu(\cdot\mid \#\t \geq \alpha p^2)$. From Theorem
\ref{mainsuper}, we now get that
\begin{equation}
\label{estBRW1}
P\Bigg(\Big|\frac{\#\mathcal{S}^{(p)}_{i_{p,1}} +\cdots+ \#\mathcal{S}^{(p)}_{i_{p,m_p}}}
{\#\t^{i_{p,1}}+\cdots + \#\t^{i_{p,m_p}}} - c_{\mu,\theta}\Big|>\ve\;
\Bigg|\;m_p\geq 1\Bigg)
\build{\la}_{p\to\infty}^{} 0.
\end{equation}
Then, on the one hand, we have from \eqref{bdrangeBRW},
$$\mathsf{R}(\z^{(p)} )\leq \#\mathcal{S}^{(p)}_{i_{p,1}} +\cdots+ \#\mathcal{S}^{(p)}_{i_{p,m_p}} + 
\sum_{i\in\{1,\ldots,p\}\backslash\{i_{p,1},\ldots,i_{p,m_p}\}}
\#\t^i,$$
and on the other hand,
$$\mathsf{R}(\z^{(p)} )\geq \#\mathcal{S}^{(p)}_{i_{p,1}} +\cdots+ \#\mathcal{S}^{(p)}_{i_{p,m_p}} -
\sum_{1\leq k<\ell\leq m_p} \#(\mathcal{S}^{(p)}_{i_{p,k}}\cap \mathcal{S}^{(p)}_{i_{p,\ell}}).
$$
Taking into account the bound \eqref{boundsizetree} and the fact
that $p^{-2}\mathsf{N}(\z^{(p)})$ converges in distribution to a positive 
random variable, we see that the first assertion of the proposition will
follow from the last two bounds and \eqref{estBRW1}, provided we can
verify that
\begin{equation}
\label{estBRW2}
\frac{1}{p^2}
\sum_{1\leq k<\ell\leq m_p} \#(\mathcal{S}^{(p)}_{i_{p,k}}\cap \mathcal{S}^{(p)}_{i_{p,\ell}}) \build{\la}_{p\to\infty}^{(P)} 0.
\end{equation}

Recall that $m_p$ converges in distribution to a finite random variable.
In order to establish \eqref{estBRW2}, it is enough to verify that,
if $\mathcal{S}^{(p),1}$, respectively $\mathcal{S}^{(p),2}$,
is the set of points visited by a random walk
indexed by a tree distributed according to $\Pi_\mu(\cdot\mid \#\t\geq \alpha p^2)$, with the spatial location of the root equal to $x_1$, resp. to $x_2$,
and if $\mathcal{S}^{(p),1}$ and $\mathcal{S}^{(p),2}$ are independent,
we have
$$\frac{1}{p^2}\, E\Big[\#(\mathcal{S}^{(p),1}\cap \mathcal{S}^{(p),2})\Big]
 \build{\la}_{p\to\infty}^{} 0.$$
 However,
$$E\Big[\#(\mathcal{S}^{(p),1}\cap \mathcal{S}^{(p),2})\Big]
=\sum_{y\in\Z^d} P(y\in \mathcal{S}^{(p),1})\,P(y\in \mathcal{S}^{(p),2})
\leq \sum_{y\in\Z^d} P(y\in \mathcal{S}^{(p),1})^2,$$
using the Cauchy-Schwarz inequality and translation invariance, which
also allows us to take $x_1=0$. By a first moment argument,
we have then
$$P(y\in \mathcal{S}^{(p),1}) \leq \frac{G_\theta(y)}{\Pi_\mu(\#\t \geq \alpha p^2)} \wedge 1 \leq (c_{(\mu)}^{-1}\sqrt{\alpha}\,p\,G_\theta(y)) \wedge 1,$$
where the constant $c_{(\mu)}>0$ depends only on $\mu$.
Here we used the classical bound 
$$\Pi_\mu(\#\t \geq k)\geq c_{(\mu)}\,k^{-1/2},\quad k\geq 1,
$$
which follows from the fact  that the distribution
of $\#\t$ under $\Pi_\mu$ coincides with the law of the first 
hitting time of $-1$ by a random walk 
on $\Z$ with
jump distribution $\nu$ started from $0$
(see the proof of Theorem \ref{mainsuper}). 
Finally, we have
$$\frac{1}{p^2}\, E\Big[\#(\mathcal{S}^{(p),1}\cap \mathcal{S}^{(p),2})] 
\leq \sum_{y\in\Z^d} (c_{(\mu)}^{-2}\alpha\,G_\theta(y)^2 )\wedge \frac{1}{p^2}$$
and the right-hand side tends to $0$ as $p\to\infty$ by dominated
convergence, noting that
$$\sum_{y\in\Z^d} G_\theta(y)^2<\infty$$
by \eqref{Greenbound}. This completes the proof of the first assertion
of the proposition.

The second assertion follows from the first one and the convergence
in distribution of
$p^{-2}\mathsf{N}(\z^{(p)})$ to $J_{1/\sigma_\mu}$. Just note that
$J_{1/\sigma_\mu}$ has the same law as $\sigma_\mu^{-2}J_1$
by scaling, and that $J_1$ is distributed as the first hitting of $1$
by a standard linear Brownian motion, whose density is as
stated in the proposition. \endproof

\medskip
We now state the result corresponding to Proposition \ref{BRW5}
in the critical dimension $d=4$. As previously, we must restrict our attention
to the geometric offspring distribution. 

\begin{proposition}
\label{BRW4}
Suppose that $d=4$, and that $\mu$ is the critical geometric
offspring distribution. Also assume that 
$\theta$ is symmetric and has small exponential moments, and set $\sigma^2=({\rm det}(M_\theta))^{1/4}$.  For every integer $p\geq 1$, let $\z^{(p)}$ be
a branching random walk
with jump distribution $\theta$ and offspring distribution $\mu$, such that
$\langle \z_0^{(p)},1\rangle = p$. Then,
$$\lim_{p\to\infty} \frac{(\log p)\,\mathsf{R}(\z^{(p)})}{\mathsf{N}(\z^{(p)})} = 8\pi^2\,\sigma^4,
\quad\hbox{in probability.}$$
Consequently,
$$ \frac{\log p}{p^2} \;\mathsf{R}(\z^{(p)}) \build{\la}_{p\to\infty}^{\rm (d)} 4\pi^2\,\sigma^4\;J\,,$$
where $J$ is as in Proposition \ref{BRW5}.
\end{proposition}

The proof of Proposition \ref{BRW4} goes along the same lines as that of
Proposition \ref{BRW5}, using now Theorem \ref{rangecritisnake} instead
of Theorem \ref{mainsuper}. A few minor modifications are needed, but
we will leave the details to the reader. 

\medskip
\noi{\it Acknowledgement.} The first author would like to thank Itai Benjamini
for suggesting the study of the range of the discrete snake a few years
ago.

\end{document}